\documentclass[11pt]{article}
\usepackage[a4paper, total={7in, 9in}]{geometry}
\usepackage{amsmath}
\usepackage{amssymb}
\usepackage{amsthm}
\usepackage{enumerate}
\usepackage{graphicx}
\usepackage{subfigure}
\usepackage{xcolor}
\usepackage{bm}
\usepackage{booktabs}
\usepackage[ruled,vlined]{algorithm2e}
\usepackage{multirow}
\newcommand{\ds}{\displaystyle}
\newcommand{\N}{{\mathbb N}}

\newcommand{\R}{{\mathbb R}}

\newcommand{\norm}[1]{\left\Vert#1\right\Vert}
\newcommand{\inner}[2]{\langle #1, #2 \rangle}

\newcommand{\set}[1]{\left\{#1\right\}}
\newcommand{\paren}[1]{\left(#1\right)}
\newcommand{\parenc}[1]{\left[#1\right]}

\newtheorem{theorem}{Theorem}
\newtheorem{lemma}{Lemma}
\newtheorem{proposition}{Proposition}

\newtheorem{corollary}{Corollary}
\newtheorem{remark}{Remark}
\newtheorem{example}{Example}

\DeclareMathOperator{\argmin}{argmin}

\title{Towards faster first order methods: A continuous-time model to interpolate between speed and function value restart}

\author{
Juan Jos\'e Maul\'en$^{1,2}$ \and
Huiyuan Guo$^{3}$ \and
Juan Peypouquet$^{3}$
}

\date{}

\begin{document}
\maketitle

\begingroup

\footnotetext[1]{\small Instituto de Ciencias de la Ingeniería, Universidad de O’Higgins, Rancagua, Chile.}
\footnotetext[2]{\small Centro de Modelamiento Matem\'atico (CNRS IRL2807), Universidad de Chile, Santiago, Chile.}
\footnotetext[3]{\small Bernoulli Institute for Mathematics, Computer Science and Artificial Intelligence, University of Groningen, The Netherlands.}

\endgroup

\maketitle

\begin{abstract}

We introduce a new restarting scheme for a continuous inertial dynamics with Hessian driven-damping, and establish a linear convergence rate for the function values along the restarted trajectories. The proposed routine is implemented without knowing the strong convexity parameter, and is a generalization of existing speed restart schemes. It interpolates between speed and function value restarts, considerably delaying the restarting time, while preserving convergence and function value decrease. Numerical experiments show an improvement in the convergence rates for both continuous-time dynamical systems, and the associated accelerated first-order algorithms derived via time discretization. \\


\end{abstract}

\section{Introduction}

Sequences generated by Nesterov's accelerated gradient method (NAG) \cite{Nesterov1983AMF} applied to a smooth convex function (or FISTA \cite{fista}, for the nonsmooth case) 
converge, in the sense of function values, at a rate of $o(\frac{1}{k^2})$, where $k$ is the number of iterations \cite{attouch2016rate}. This is faster than standard gradient descent, whose convergence rate is $o(\frac{1}{k})$. The key to this improvement is an {\it iteration-dependent} extrapolation coefficient that modulates the momentum. For strongly convex functions, using a {\it constant} extrapolation coefficient adapted to the specific strong convexity parameter (NAG-SC), one obtains a linear convergence rate \cite{nesterov2003introductory}, which is considerably faster than the one achieved by gradient descent (see also \cite{wang2025acceleratedgradientmethodsinertial} for a more general range of parameters and a discussion thereof). Not having access to the strong convexity parameter, one can still use the original version, obtaining a linear convergence rate which is comparable to that of gradient descent \cite{li2024shi}.

In the search for a deeper understanding of (NAG), a second-order ordinary differential equation with an asymptotically vanishing damping term
\begin{equation}\label{eq:avd}\tag{AVD}
    \ddot{x}(t) + \frac{\alpha}{t}\dot{x}(t) + \nabla \phi(x(t)) =0
\end{equation}
which provides a continuous-time model for (NAG), was studied in \cite{SBC2016}. For $\alpha>3$, each trajectory generated by such a differential equation converges to a minimizer, strongly if the function is strongly convex, weakly in general \cite{APR2018}. The function values converge at a rate of $o\left(\frac{1}{t^2}\right)$ for convex functions \cite{May2017}. 

In contrast with (NAG), the convergence is {\it not} linear for strongly convex functions. Instead, it is capped at $\mathcal O\left(\frac{1}{t^3}\right)$, even if the function involved is quadratic. In view of this, the authors of \cite{SBC2016} developed a {\it speed restarting} scheme that guarantees linear convergence of the function values for strongly convex functions. Roughly speaking, whenever the speed of the trajectory ceases to increase, one stops, resets the time to zero and uses the current state as the new starting point. 
A rather obvious alternative is to restart the process when the function values stop decreasing (see \cite{o2015adaptive} for the algorithmic version). Unfortunately, although progress has been made in this direction \cite{bao2024accelerated,giselsson2014monotonicity,lin2014adaptive}, fully establishing a linear convergence rate for this restarting policy remains as an open problem. Another approach is to apply restarts at a fixed intervals (or number of iterations), strategy introduced in \cite{necoara2019linear,nesterov2013gradient}. However, these relies on accurate knowledge of certain parameters, such as the strong convexity constant or the optimal function value, which makes them unpractical for applications. Estimation tools for this have been proposed in \cite{alamo2022restart,aujol2025fista,roulet2017sharpness}. Other restarting schemes are specifically tailored to FISTA \cite{alamo2019gradient,alamo2019restart,aujol2025fista,fercoq2019adaptive}, primal-dual splitting algorithms \cite{applegate2023faster}, Schwarz methods \cite{park2021accelerated}, stochastic gradient descent \cite{wang2022scheduled}, or also for nonconvex  \cite{li2023restarted} and multi-objective optimization problems \cite{luo2025accelerated}.  \\

One the other hand, the solutions of \eqref{eq:avd}, as well as the sequences generated by (NAG), show heavy oscillations of the function values. To mitigate this effect, a variant with a Hessian-driven damping term was introduced in \cite{attouch2016fast}, based on \cite{AABR2002}, namely:
\begin{equation}\label{eq:din_avd}\tag{DIN-AVD}
    \ddot{x}(t) + \frac{\alpha}{t}\dot{x}(t) + \beta \nabla^2 \phi(x(t))\dot{x}(t) + \nabla \phi(x(t)) =0.
\end{equation}
Surprisingly, by adding a time-dependent coefficient in the gradient term, one recovers a {\it high-resolution} approximation of (NAG) \cite{shi2022understanding}:
\begin{equation}\label{eq:hr_dinavd}\tag{HR-DIN-AVD}
  \ddot{x}(t) + \frac{\alpha}{t}\dot{x}(t) + \beta \nabla^2 \phi(x(t))\dot{x}(t) + \left( \gamma + \frac{r}{t}\right)\nabla \phi(x(t)) =0.  
\end{equation}
The asymptotic properties of \eqref{eq:hr_dinavd} have been studied in \cite{attouch2020first,shi2022understanding}. Notably, linear convergence of the function values in the strongly convex case has been established in \cite{li2024linear,wang2025fast}. These convergence results rely on assumptions on the parameters that depend on the strong convexity constant. It is not known whether linear convergence holds for \eqref{eq:din_avd}.

Besides the theoretical results on convergence and stabilization obtained in \cite{attouch2016fast}, simulations suggest that, although the solutions of \eqref{eq:din_avd} are {\it lazier} than those of \eqref{eq:avd} for small $t$, they end up taking over and remain more stable. A simple example can be illustrative: 

\begin{example}\label{ex:intr}
Consider the function $\phi:\R^3\to\R$, defined by 
\begin{equation}\label{eq:fun_example}
    \phi(x) = \frac{1}{2}\paren{x_1^2+\rho x_2^2+\rho^2 x_3^2},
\end{equation}
with $\rho>1$. The solutions of \eqref{eq:din_avd} present an oscillatory behavior, that are tempered in the presence of the Hessian-driven damping when $\beta>0$, as shown in Figure \ref{fig:comp_DINAVD}. Solutions of \eqref{eq:hr_dinavd} are also depicted, where a similar behavior as \eqref{eq:din_avd} can be observed.
\end{example}

\begin{figure}[h]
    \centering
    \subfigure[$\rho=1$]{\includegraphics[width=0.48\textwidth]{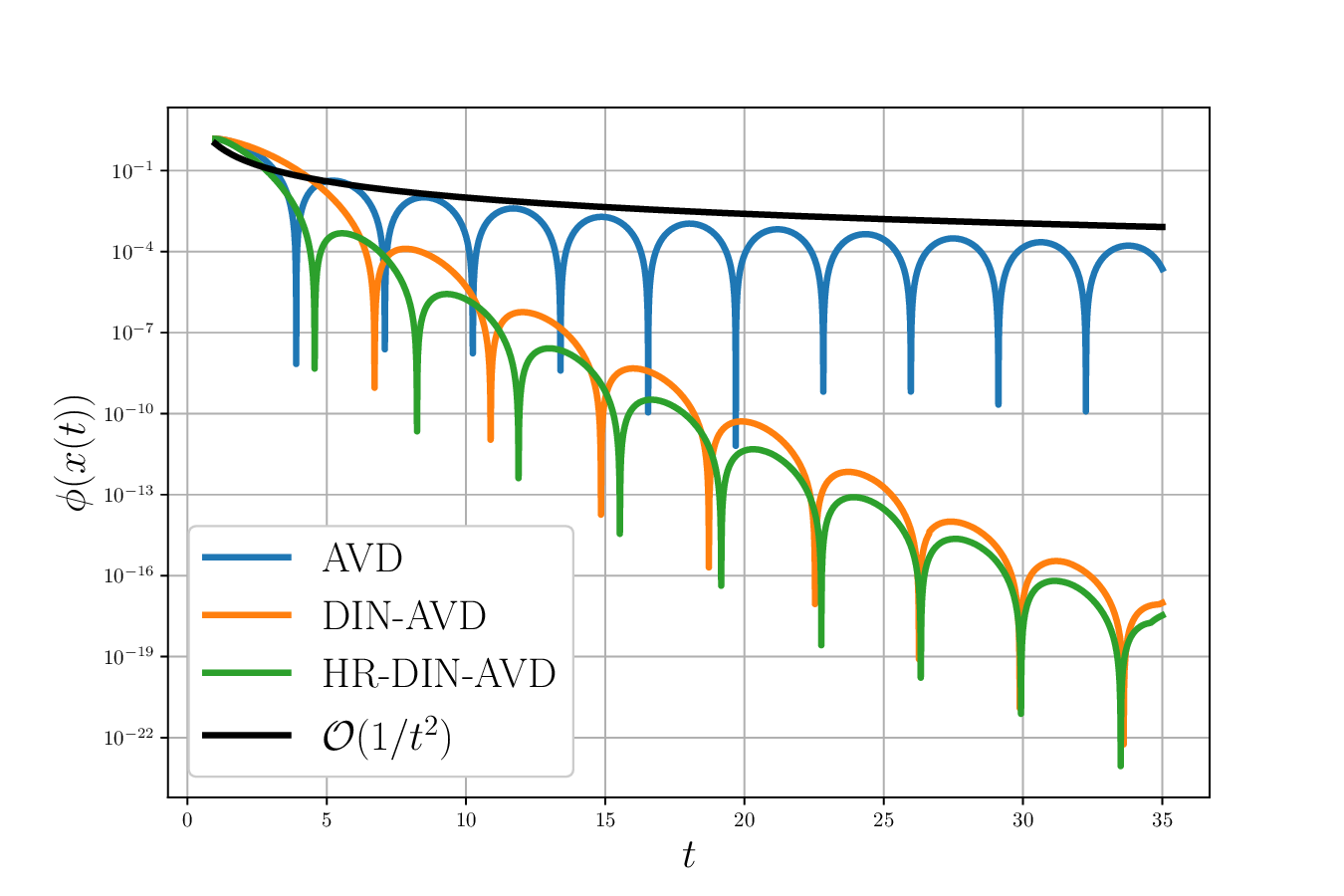}}
    \subfigure[$\rho=10$]{\includegraphics[width=0.48\textwidth]{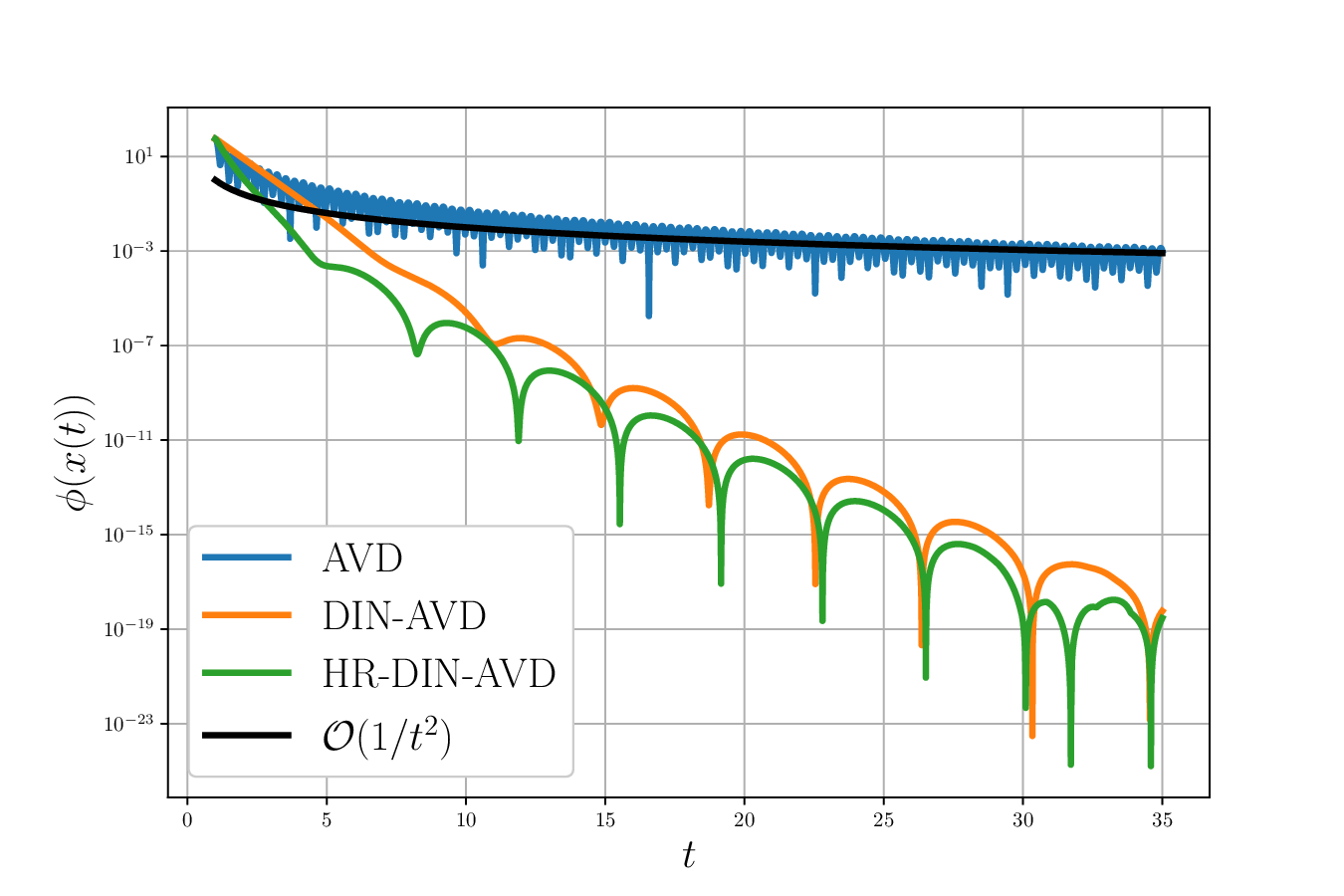}}
    \caption{Function values along the solutions of \eqref{eq:din_avd} for the function $\phi$ in Example \ref{ex:intr}, with $\alpha=3$, $\beta=1$, initial condition $x(1)=(1,1,1)$ and $\dot{x}(1)=0$, on the interval $[1,35]$. For \eqref{eq:hr_dinavd}, we consider $\gamma=1$ and $r=\frac{\alpha\beta}{2}$ to fulfill the hypotheses in \cite{wang2025fast}.}
    \label{fig:comp_DINAVD}
\end{figure}

Now, when the speed restarting scheme is applied to \eqref{eq:din_avd}, linear convergence of the function values is still guaranteed \cite{maulen2023speed}, even when the objective function is not strongly convex but has quadratic growth. Moreover, numerical simulations show that the combination of restart {\it and} Hessian-driven damping tends to yield better performance than each technique applied alone. The case of constant coefficients, in line with (NAG-SC), was studied in \cite{guo2024speed}. Unlike (NAG-SC), one can get linear convergence rate without knowing the strong convexity parameter $\mu$, by applying speed restart scheme.  \\

The aforementioned works on speed restart provide theoretical linear convergence rates, but the restart occurs significantly earlier than for the function values restart. Although there is no convergence guarantee for the latter, it does seem to be a competitive heuristic. In order to seize this improvement opportunity, we propose an {\it extended} speed restarting scheme, whose restarting time is closer to the function value restarting time while preserving the linear convergence guarantee without knowing the strong convexity parameter $\mu$.

The paper is organized as follows. In Section \ref{sec:restart_main} we recall some properties of the solutions of \eqref{eq:din_avd}, mostly borrowed from \cite{maulen2023speed}. The proposed restarting scheme is described in Section \ref{sec:thetrue_restartmain}, together with the precise statement of our main result, namely Theorem \ref{T:TheTheorem}, which guarantees the linear convergence of the function values on the restarted trajectories. Lower and upper bounds for the extended restarting times are established in Sections \ref{sec:lower_bound} and \ref{sec:upper_bound}, respectively. In Section \ref{sec:main_theo} we quantify the function value decrease between restarts, and complete the proof of Theorem \ref{T:TheTheorem}. Section \ref{sec:numerical} contains numerical experiments that illustrate how our extended speed restarting scheme improves the convergence of \eqref{eq:din_avd}, and how, used as a heuristic, it enhances the performance of the corresponding restarted algorithms obtained via time discretization. Beyond standard convex optimization benchmarks, we also present an application to kernel regularized learning. This example illustrates how the proposed methodology, originally developed in a purely theoretical framework, can be naturally incorporated into algorithms arising in machine learning. Finally, conclusions and open questions are discussed in Section \ref{sec:conculsion}.

\section{The solutions of \eqref{eq:din_avd}}\label{sec:restart_main}

Throughout this paper, $\phi:\R^n \to \R$ is a twice continuously differentiable convex function, which attains its minimum $\phi^*$ at $x^*$. We assume, moreover, that $\phi$ is \textit{strongly convex} with parameter $\mu > 0$, and that its gradient $\nabla\phi$ is Lipschitz continuous with constant $L>0$. This implies that
\begin{equation}\label{eq:strong1}
    2\mu(\phi(x) - \phi^*) \leq \norm{\nabla \phi(x)}^2
\end{equation}
and
\begin{equation}\label{eq:strong2}
    \langle\nabla^2\phi(x)v,v\rangle \ge\mu\|v\|^2,
\end{equation}
for all $x,v \in \R^n$. \\

For every $z\in\R^n$, there is a unique function $x_z \in \mathcal{C}^2\paren{\paren{0,+\infty};\R^n}\cap \mathcal{C}^1\paren{[0,+\infty);\R^n}$ such that $x_z(0)=z$, $\dot{x}_z(0)=0$ and $x=x_z$ satisfies \eqref{eq:din_avd} for every $t>0$ (see \cite[Theorem 2.1]{maulen2023speed}). Such a function is the (lazy) \textit{solution} or \textit{trajectory} of \eqref{eq:din_avd} starting from $z$. \\

Now, by rewriting equation \eqref{eq:din_avd} as
$$\dfrac{d}{dt}(t^\alpha\dot{x}(t))  = - \beta t^\alpha\nabla^2 \phi(x(t))\dot{x}(t) - t^\alpha\nabla \phi(x(t)),$$
and integrating over $[0,t]$, to get 
\begin{align}\label{eq:eqI} 
-t^\alpha\dot{x}(t) & = \beta   \int_{0}^{t}u^\alpha\nabla^2 \phi(x(u))\dot{x}(u)\, du + \int_{0}^{t}u^\alpha \nabla \phi(x(u))\, du \nonumber \\
& = I_z(t) + \dfrac{ t^{\alpha+1}}{\alpha+1}\nabla \phi(z),
\end{align}
where 
\begin{equation}\label{E:def_I}
    I_z(t):=\beta   \int_{0}^{t}u^\alpha\nabla^2 \phi(x(u))\dot{x}(u)\, du + \int_{0}^{t} u^\alpha(\nabla \phi(x(u))- \nabla \phi(z))\, du.
\end{equation}

It is possible to estimate $\|\dot x_z(t)\|$, $\|I_z(t)\|$ and $\dot I_z(t)$ in terms of $\norm{\nabla \phi (z)}$ and some functions of $t$, which are  {\it universal}, in the sense that they do not depend on the function $\phi$ or the initial state $z$. Following \cite{maulen2023speed}, set
\begin{equation} \label{E:H}
    H(t)=1 -\dfrac{\beta Lt}{(\alpha+2)} - \dfrac{ Lt^2}{2(\alpha+3)}.
\end{equation}
The function $H$ is concave, quadratic, does not depend on $z$, and has exactly one positive zero, given by
\begin{equation} \label{E:tau1}
    \tau_1=- \left(\dfrac{ \alpha+3}{\alpha+2}\right)\beta+\sqrt{\left(\dfrac{ \alpha+3}{\alpha+2}\right)^2\beta^2 + \dfrac{2(\alpha+3)}{L}}.
\end{equation}
In particular, $H$ decreases strictly from $1$ to $0$ on $[0,\tau_1]$, and there is exactly one value $\tau_2 \in (0,\tau_1)$, namely
\begin{equation}\label{E:tau2}
    \tau_2 = -\paren{\dfrac{\alpha+3}{\alpha+2}}\beta+ \sqrt{\paren{\dfrac{\alpha+3}{\alpha+2}}^2\beta^{2}+\dfrac{\alpha+3}{L}},
\end{equation}
for which $H(\tau_2) = 1/2$. The following result gathers some estimations from \cite{maulen2023speed}, that will prove useful for our purpose:

\begin{proposition} \label{P:I_and_x}
For every $t \in (0,\tau_1)$, we have
\begin{eqnarray*}
    (\alpha+1)H(t)\|\dot x_z(t)\| 
    & \le & t\norm{\nabla \phi (z)} \\
    (\alpha+1)H(t)\norm{I_z(t)}  
    & \le & t^{\alpha+1}\big(1-H(t)\big) \norm{\nabla \phi (z)} \\
    (\alpha+1)H(t)\left\|\dot I_z(t)\right\|
     & \le & t^{\alpha+1}\left[\beta L+\dfrac{ Lt}{2}\right] \norm{\nabla \phi (z)}
\end{eqnarray*}
Moreover, for every $t \in (0,\tau_2)$, we have
\begin{equation} \label{E:reverse}
t\big(2H(t)-1\big)\norm{\nabla\phi(z)}\le (\alpha+1)H(t)\norm{\dot{x}_z(t)}.    
\end{equation}
\end{proposition}

\section{Restarting schemes} \label{sec:thetrue_restartmain}

A {\it restarting scheme} is a mechanism that consists in stopping the evolution of a time process, whenever a specified criterion is met, and then using its current state as the initial condition for a new process. The pieces obtained inductively by this procedure are then concatenated to define a process over the whole time horizon. \\

Given $z\in\R^n$, the {\it restarting time} $T(z)$ is the first time the (yet to be specified) criterion is satisfied for the process $X_z$ starting from $z$. The {\it restarted trajectory} starting from $z_0\in \R^n$ is the function $\bm{x}_{z_0}:[0,\infty)\to\R^n$ defined inductively as follows:
\begin{enumerate}
    \item Compute $T_1=T(z_0)$, and set $\bm{x}_{z_0}(t)=X_{z_0}(t)$, for $t\in[0,T_1]$, as well as $z_1=\bm{x}_{z_0}(T_1)$.
    \item For $k\ge 1$, having defined $T_k$ and $z_k$, compute $T_{k+1}=T_k+T(z_k)$, in order to define $\bm{x}_{z_0}(t)=X_{z_k}(t-T_k)$, for $t\in[T_k,T_{k+1}]$, as well as $z_{k+1}=\bm{x}_{z_0}(T_{k+1})$.
\end{enumerate}

\subsection{From restarting schemes to linear convergence} \label{SS:schemes}

A restarting scheme for \eqref{eq:din_avd} is {\it proper} if 
\begin{itemize} 
    \item [i)] there exist $\tau_*,\tau^*>0$ such that $\tau_*\le T(z) \leq \tau^*$, for every $z\notin\argmin(\phi)$; 
    \item [ii)] $t\mapsto \phi\big(x_z(t)\big)$ is nonincreasing on $[0,T(z)]$; and 
    \item[iii)] there is $Q\in(0,1)$ such that $\phi\big(x_z(T(z))\big)-\phi^* \le Q\big(\phi(z)-\phi^*\big)$
    for every $z\notin\argmin(\phi)$. 
\end{itemize}

The {\it speed restarting time}, introduced in \cite{SBC2016} for $(\alpha,\beta)=(3,0)$, and extended in \cite{maulen2023speed} to $(\alpha,\beta)\in(0,\infty)\times[0,\infty)$, is proper. The restarting time is given by
$$T^0(z) = \inf\bigg\{ t >0\, :\, \frac{1}{2}\frac{d}{dt}\|\dot x_z(t)\|^2 \le 0 \bigg\}.$$  

In a proper restarting scheme, the restarting time cannot be larger than the {\it function value restarting time}
$$T^{\rm fv}(z) = \inf \set{t > 0 \, : \, \dfrac{d}{dt}\phi(x_z(t))\geq 0}.$$

The relevance of proper schemes is explained by the following result, which gathers the essence of the strategy used in \cite{SBC2016,maulen2023speed} to establish the linear convergence of speed restarting schemes:

\begin{lemma} \label{L:abstract_linear}
Consider a proper restarting scheme for \eqref{eq:din_avd}, and let $\bm{x}_{z_0}$ be the restarted trajectory with initial state $z_0$. For every $t>0$, we have
$$\phi(\bm{x}_{z_0}(t)) - \phi^* \leq Ce^{-\kappa t}\big(\phi(z_0)-\phi^*\big),$$
where $C = Q^{-1}$ and $\kappa = -\frac{1}{\tau^*}\ln(Q)>0$.
\end{lemma}

\begin{proof}
The existence of the lower bound $\tau_*$ guarantees that $\bm{x}_{z_0}(t)$ is defined for all $t>0$. Now, given $t>0$, let $m$ be the largest positive integer such that $m\tau^* \leq t$. By time $t$, the trajectory will have been restarted at least $m$ times. By definition, $(m+1)\tau^* > t$, which gives $m > \frac{t}{\tau^*} -1$. Let $\tau_t$ be the largest restarting time such that $\tau_t\le t$ (it is {\it at least} the $m$-th). Using property ii) and, inductively applying property iii), we deduce that 
\[\phi(\bm{x}_{z_0}(t)) - \phi^*\leq\phi\big(\bm{x}_{z_0}(\tau_t)\big)  - \phi^* \leq Q^m (\phi(z_0)-\phi^*) \leq Q^{\frac{t}{\tau^*} -1}(\phi(z_0)-\phi^*),\]
as claimed. 
\end{proof}

The numerical evidence in \cite{maulen2023speed,SBC2016} reveals that the speed restarting time can be considerably shorter than the function value restart. In some cases, this may hinder the effectiveness of the method, and limit its applicability to discrete-time algorithms. In order to delay the restarting time, we propose a procedure that aims to interpolate between the {\it vanilla} speed restart and the function value restart (which it precisely does, when $\beta=0$).



 

\subsection{An extended speed restart for \eqref{eq:din_avd}}

Let $x_z$ be the solution of \eqref{eq:din_avd} starting from $z\notin\argmin(\phi)$. For $t>0$ and $\lambda\in[0,1]$, set
$$\varphi_{z, \lambda}(t)=\dfrac{1}{2}\dfrac{d}{dt} \norm{\dot{x}_z(t)}^2+\lambda\dfrac{\alpha}{t}\norm{\dot{x}_{z}(t)}^2.$$
The ($\lambda$-extended) {\it speed restarting time} is
\begin{equation} \label{E:ESRestart}
T^\lambda(z) = \inf\bigg\{ t >0\, :\, \varphi_{z, \lambda}(t) \le 0 \bigg\}.    
\end{equation}
Using \eqref{eq:din_avd}, we have
\begin{eqnarray}
\dfrac{d}{dt}\phi\big(x_z(t)\big) + \varphi_{z, \lambda}(t) 
& = & \inner{\nabla \phi\big(x_z(t)\big)}{\dot{x}_z(t)} +  \inner{\ddot{x}_z(t)}{\dot{x}_z(t)} + \lambda\dfrac{\alpha}{t}\norm{\dot{x}_z(t)}^2  \nonumber \\
& = & -\dfrac{\alpha(1-\lambda)}{t}\norm{\dot{x}_z(t)}^2 \nonumber \\
&& -\beta\langle\nabla^2\phi\big(x_z(t)\big)\dot x_z(t),\dot x_z(t)\rangle \label{E:restarting_times_1}\\
& \le & -\left[\dfrac{\alpha(1-\lambda)}{t} + \beta\mu\right]\norm{\dot{x}_z(t)}^2 \label{E:restarting_times_2}\\
& \le & 0. \nonumber 
\end{eqnarray}

When $\beta=0$, by setting $\lambda=1$ and using \eqref{E:restarting_times_1}, we see that
$$\dfrac{d}{dt}\phi\big(x_z(t)\big) + \varphi_{z, 1}(t)\equiv 0,$$
and so $T^1(z)=T^{\rm fv}(z)$, which makes $T^\lambda$ an interpolation between $T^0$ and $T^{\rm fv}$. In general, for every $z\in\R^n$, we have
$$T^0(z) \le T^\lambda(z) \le T^{\rm fv}(z),$$
where the second inequality is an immediate consequence of \eqref{E:restarting_times_2}. \\ 

Figure \ref{fig:twoaxis} displays the values of $\varphi_{z,\lambda}(t)$ and $\phi(x_z(t))$, where $x_z(t)$ is solution of \eqref{eq:din_avd}, where $\alpha=3$, $\beta\in\{0,1\}$, and the function $\phi$ is that of Example \ref{ex:intr} with $\rho=1$. One can clearly appreciate the delays the activation of the $\lambda$-extended speed restart, as $\lambda$ increases.

\begin{figure}[h]
    \subfigure[$\beta=0$]{\includegraphics[width=0.45\textwidth]{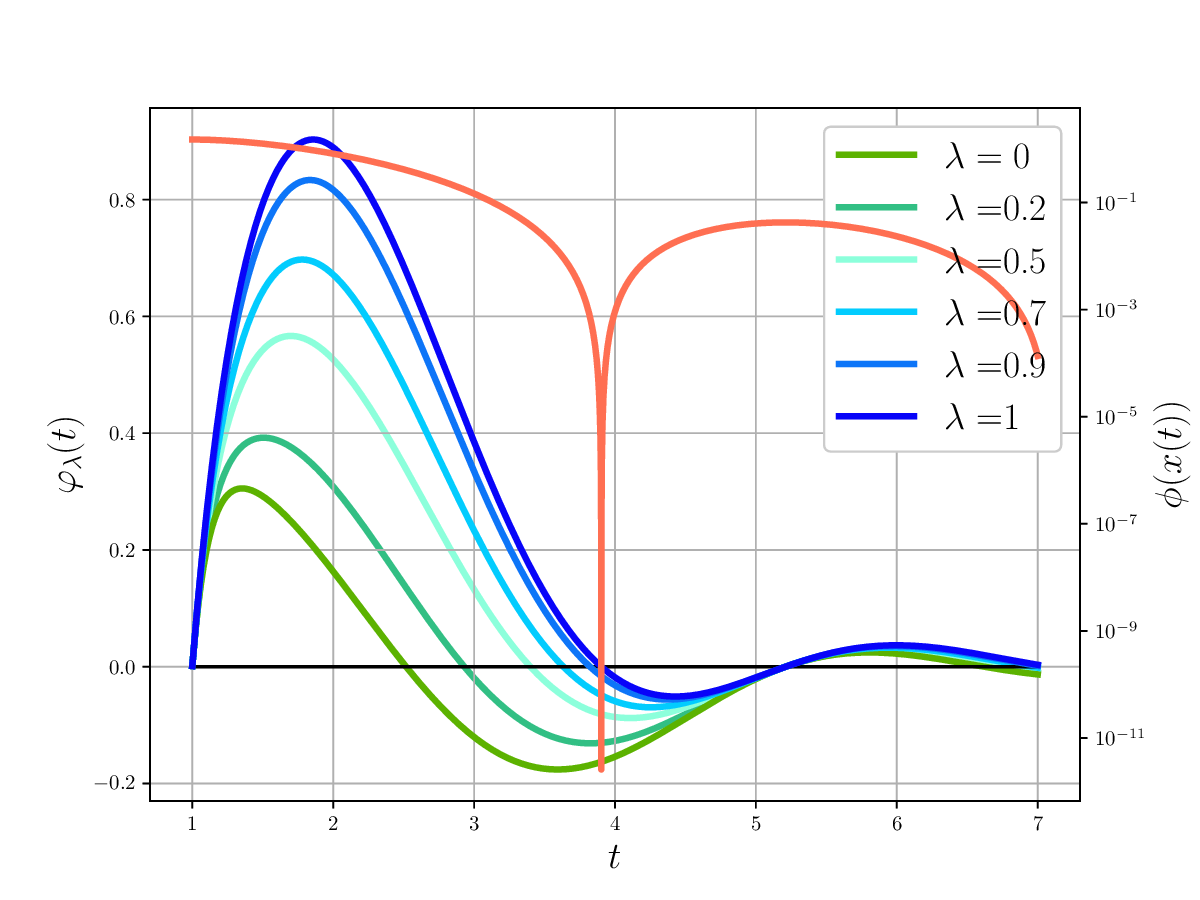}}
    \subfigure[$\beta=1$]{\includegraphics[width=0.45\textwidth]{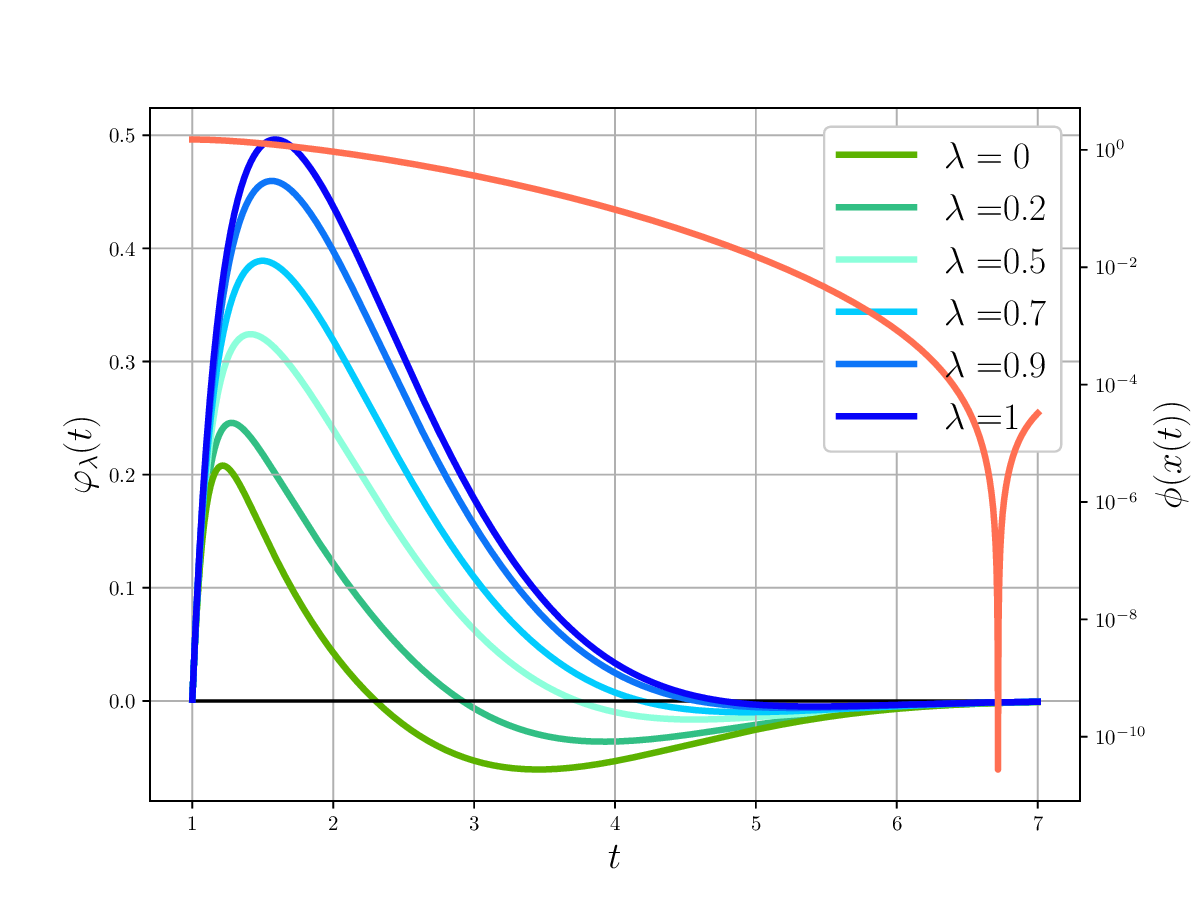}}
    \caption{In blue-green, we show the graph of $\varphi_{z,\lambda}(t)$ in the context of Example \ref{ex:intr}, with  $\rho=1$, $z=(1,1,1)$, $\alpha=3$. The function values $\phi(x_z(t))$ are displayed in red. }
    \label{fig:twoaxis}
\end{figure}

Our main result is the following:

\begin{theorem} \label{T:TheTheorem}
Let $\phi:\R^n \to \R$ be a twice continuously differentiable $\mu$-strongly convex function which attains its minimum value $\phi^*$ at $x^*\in\R^n$. Assume that $\nabla \phi$ is $L$-Lipschitz continuous, and fix $\alpha,\beta > 0$ and $\lambda \in [0,\frac{1}{2\alpha}]$. There exist $C\ge 0$ and $\kappa>0$ such that, for every $z_0 \in \R^n$ and $t>0$, the restarted trajectory $\bm{x}_{z_0}$ be  defined by \eqref{eq:din_avd}, starting from $z_0$, satisfies
\[\phi\big(\bm{x}_{z_0}(t)\big) - \phi^* \leq Ce^{-\kappa t}(\phi(z_0) - \phi^*).\]
\end{theorem}

In view of Lemma \ref{L:abstract_linear}, it suffices to prove that the $\lambda$-extended speed restart is proper, which we do in the next few sections, culminating in Proposition \ref{P:conclusion}. Property ii) is a consequence of the fact that $T^\lambda \le T^{\rm fv}$. On the other hand, since $T^0 \le T^\lambda$ and the speed restart is proper, the existence of $\tau_*$ in property i) is guaranteed. However, the purpose of introducing this new scheme is to delay the restarting time. The quantification of this delay is of the utmost relevance, and the object of the next section.

\section{Lower bound for the restarting time} \label{sec:lower_bound}

The purpose of this section is to approximate
$$\bm T_*  = \inf\set{T^\lambda(z) : z \notin \argmin(\phi)},$$
for which we first establish the following estimation:

\begin{lemma} \label{lem:technical}
    Let $x_z$ be the lazy solution of \eqref{eq:din_avd} starting from $z\notin\argmin(\phi)$, with $\alpha>0$ and $\beta\ge 0$. For every $\lambda \in [0,1]$ and $t\in(0,\tau_1)$, we have
    $$
    \inner{\dot{x}_z(t)}{\ddot{x}_z(t)} + \lambda\frac{\alpha}{t}\|\dot x_z(t)\|^2
    \ge 
    G(t)\frac{t\,\|\nabla \phi(z)\|^2}{(\alpha+1)^2H(t)^2}.
    $$
where 
\begin{eqnarray*}
    G(t)
&= & (1+\alpha\lambda)H(t)^2
-\alpha(1-\lambda)\big(1-H(t)\big)^2  
- \beta Lt-\dfrac{ Lt^2}{2} \\
&& - \big|1+\alpha(2\lambda-1)\big|H(t)\big(1-H(t)\big).
\end{eqnarray*}
\end{lemma}

\begin{proof}
From \eqref{eq:eqI}, we know that 
\begin{equation}\label{eq:dotx}
 \dot{x}_z(t)= -\dfrac{1}{t^\alpha}I_z(t) -  \dfrac{ t}{\alpha+1}\nabla \phi (z),
\end{equation}
whence
$$\ddot x_z(t)
=-\dfrac{d}{dt}\left[ \dfrac{1}{t^\alpha}I_z(t)\right] -\dfrac{1}{\alpha+1}\nabla \phi (z) 
= \dfrac{\alpha}{t^{\alpha+1}} I_z(t) - \dfrac{1}{t^\alpha}\dot I_z(t)-\dfrac{1}{\alpha+1}\nabla \phi (z).$$
It follows that
\begin{eqnarray*}
\langle\ddot x_z(t),\dot x_z(t)\rangle 
& = & - \left\langle \dfrac{\alpha}{t^{\alpha+1}} I_z(t) -\dfrac{1}{\alpha+1}\nabla \phi (z), \dfrac{1}{t^\alpha}I_z(t) + \dfrac{ t}{\alpha+1}\nabla \phi (z)\right\rangle \\ 
&&\ - \dfrac{1}{t^\alpha}\langle \dot I_z(t),\dot x_z(t)\rangle  \\
& = & -\frac{\alpha}{t^{2\alpha+1}}\|I_z(t)\|^2 
+ \frac{ t\|\nabla\phi(z)\|^2}{(\alpha+1)^2}
+ \frac{(1-\alpha)}{(\alpha+1)t^\alpha}\langle I_z(t),\nabla\phi(z)\rangle \\ 
&&\ 
- \dfrac{1}{t^\alpha}\langle \dot I_z(t),\dot x_z(t)\rangle. \\
\lambda\frac{\alpha}{t}\|\dot x_z(t)\|^2
& = & \lambda \left[ \frac{\alpha}{t^{2\alpha+1}}\|I_z(t)\|^2+\frac{\alpha t\|\nabla \phi(z)\|^2}{(\alpha+1)^2}+\frac{2\alpha}{(\alpha+1)t^\alpha}\langle I_z(t),\nabla \phi(z)\rangle\right].
\end{eqnarray*}
Adding these two equalities, we obtain
\begin{eqnarray*}
\inner{\ddot{x}_z(t)}{\dot{x}_z(t)} + \lambda\frac{\alpha}{t}\|\dot x_z(t)\|^2 & = & \dfrac{\alpha(\lambda-1)}{t^{2\alpha+1}}\norm{I_z(t)}^2 + \dfrac{(\lambda\alpha+1)}{(\alpha+1)^2}t \norm{\nabla \phi(z)}^2  \\
& & - \dfrac{1}{t^\alpha}\inner{\dot{I}_z(t)}{\dot{x}_z(t)} \\
& & +\  \dfrac{1}{(\alpha+1)t^\alpha}(1+\alpha(2\lambda-1))\inner{I_z(t)}{\nabla\phi(z)} \\
& \ge & \dfrac{\alpha(\lambda-1)}{t^{2\alpha+1}}\norm{I_z(t)}^2 + \dfrac{(\lambda\alpha+1)}{(\alpha+1)^2}t \norm{\nabla \phi(z)}^2  \\
& & - \dfrac{1}{t^\alpha}\norm{\dot{I}_z(t)}\norm{\dot{x}_z(t)} \\
&& -\ \dfrac{1}{(\alpha+1)t^\alpha}\big|1+\alpha(2\lambda-1)\big|\norm{I_z(t)}\norm{\nabla\phi(z)}.
\end{eqnarray*}
Then, we use the first three estimations in Proposition \ref{P:I_and_x} to bound the terms on the right-hand side, obtaining
\begin{eqnarray*}
   \inner{\ddot{x}_z(t)}{\dot{x}_z(t)} + \lambda\frac{\alpha}{t}\|\dot x_z(t)\|^2 & \geq & 
   G(t)\dfrac{ t\norm{\nabla \phi (z)}^2}{(\alpha+1)^2H(t)^2},
\end{eqnarray*}
where 
\begin{eqnarray*}
    G(t)
&= & (1+\alpha\lambda)H(t)^2
-\alpha(1-\lambda)\big(1-H(t)\big)^2  
- \beta Lt-\dfrac{ Lt^2}{2} \\
&& - \big|1+\alpha(2\lambda-1)\big|H(t)\big(1-H(t)\big),
\end{eqnarray*}
as stated.
\end{proof}

Since $G(0)=1+\alpha\lambda>0$, Lemma \ref{lem:technical} shows that the restart cannot occur before the first zero of $G$, which must be located before $\tau_2$ because
\begin{align*}
    G(\tau_2)&<\frac{1}{4}\left[1+\alpha\lambda-\alpha(1-\lambda)-\big|1+\alpha(2\lambda-1)\big|\right]\\
    &=\frac{1}{4}\left[1+\alpha(2\lambda-1)-\big|1+\alpha(2\lambda-1)\big|\right]\\
    &\le 0.
\end{align*}
Moreover, the function $G$ is strictly decreasing on $(0,\tau_2)$. Indeed,
\begin{eqnarray*}
G'(t)
& = & 2(1+\alpha\lambda)H(t)H'(t)
+2\alpha(1-\lambda)\big(1-H(t)\big)H'(t)   \\
&& -\ \beta L-Lt -\big|1+\alpha(2\lambda-1)\big|H'(t)\big(1-2H(t)\big) \\
& \le & \Big[2(1+\alpha\lambda)H(t)
+2\alpha(1-\lambda)\big(1-H(t)\big) \Big]H'(t) \\
& = & \Big[2\alpha(1-\lambda)+2\left[1+\alpha(2\lambda-1)\right]H(t)\Big]H'(t)    
\end{eqnarray*}
because $2H(t)\ge 1$ on  $(0,\tau_2)$, and $H'(t)<0$ on $(0,\infty)$. If $1+\alpha(2\lambda-1)\le 0$, we use the fact that $H(t)\le 1$ to get
$$2\alpha(1-\lambda)+2\left[1+\alpha(2\lambda-1)\right]H(t) 
\ge 2\alpha-2\alpha\lambda+2+4\alpha\lambda-2\alpha 
= 2(1+\alpha\lambda),$$
and so $G'(t)<0$. On the other hand, if $1+\alpha(2\lambda-1)\ge 0$, then
$$2\alpha(1-\lambda)+2\left[1+\alpha(2\lambda-1)\right]H(t)  
 \ge 2\alpha(1-\lambda),$$
and $G'(t)<0$ as well. As a consequence, there is a unique $\tau_3\in (0,\tau_2)$ such that $G(\tau_3)=0$, and $G(t)>0$ on $[0,\tau_3)$. From Lemma \ref{lem:technical}, we obtain:

\begin{corollary} \label{C:lower_bound}
For any $\lambda \in [0,1]$, $\alpha>0$, $\beta \geq 0$, we have $\tau_1 > \tau_2 > \tau_3 >0$ and $\tau_3 \leq \bm T_*$. 
\end{corollary}

\subsection{A few remarks on the values of $\tau_3$}

In general, $\tau_3$ is the first positive root of a polynomial of degree at most four. For certain values of the parameters, $\tau_3$ admits a relatively simple expression that is useful to compare, for instance, with the values obtained in \cite{maulen2023speed,SBC2016}. \\

\noindent{\bf Case I.} If $\alpha\ge 1+2\alpha\lambda$, we have
$$G(t) = 1 + \alpha\lambda  - \beta L t \paren{\dfrac{2\alpha+3}{\alpha+2}} -  L t^2 \paren{\dfrac{\alpha+2}{\alpha+3}},$$ 
and so
$$\tau_3 =  -\dfrac{(2\alpha+3)(\alpha+3)}{2(\alpha+2)^2}\beta + \sqrt{\dfrac{(2\alpha+3)^2(\alpha+3)^2}{4(\alpha+2)^4}\beta^2+\dfrac{(1+\alpha\lambda)(\alpha+3)}{ L(\alpha+2)}}.  
$$

The combination $(\alpha,\beta,\lambda)=(3,0,0)$ was analyzed in \cite{SBC2016}, where they provide $\frac{4}{5\sqrt{L}}$. Our formula gives $\tau_3=\sqrt{\frac{6}{5L}}$, which is larger by a factor of $\sqrt{\frac{15}{8}}\sim 1.37$. Keeping $(\alpha,\beta)=(3,0)$, choosing $\lambda=\frac{1}{2\alpha}=\frac{1}{6}$ gives $\tau_3= \frac{3}{\sqrt{5L}}$, which is larger than the value in \cite{SBC2016} by a factor of 
$\frac{3\sqrt{5}}{4}\sim 1.68$. Notice also that $\lambda=1$ gives $\tau_3=\sqrt{\frac{3}{L}}$, which is larger by a factor of $\frac{5\sqrt{3}}{4}\sim 2.17$. \\

\noindent{\bf Case II.} If $\alpha < 1+2\alpha\lambda$ , the case $\beta=0$ is still simple. We get 
\[G(t) = 1 + \alpha\lambda - \dfrac{2\alpha\lambda+3}{\alpha+3}Lt^2 + \dfrac{2\alpha\lambda - \alpha +1}{2(\alpha+3)^2}L^2t^4,\]
thus
$$\tau_3=\sqrt{\frac{\alpha+3}{(2\alpha\lambda-\alpha+1)L}\left(2\alpha\lambda+3-\sqrt{2\alpha^2\lambda+6\alpha\lambda+2\alpha+7}\right)}.$$
A special case, whose relevance will become apparent in the next section, is when $\alpha=\frac{1}{2}$ and $\lambda=1$, which gives
$$\tau_3=\sqrt{\frac{7(8-\sqrt{46})}{6L}}\sim \frac{1.19}{\sqrt{L}}.$$
On the other hand, for $\alpha=1$ and $\lambda=\frac{1}{2}$, we get $$\tau_3=2\sqrt{\frac{4-\sqrt{13}}{L}}\sim\frac{1.26}{\sqrt{L}}.$$





\section{Upper bound for the restarting time} \label{sec:upper_bound}

We now turn our attention to 
\[\bm T^* := \sup\set{T^\lambda(z): \, z \notin \argmin(\phi)},\]
for which we can prove the following:

\begin{proposition}\label{p:upperbound_existence} 
Let $x_z$ be the lazy solution of \eqref{eq:din_avd} starting from $z \notin \argmin(\phi)$. Let $\phi$ be $\mu$-strongly convex. For each $\lambda \in \parenc{0,\frac{1}{2\alpha}}$, $\bm T^*<\infty$.
\end{proposition}

\begin{proof} Write $g(t) = \norm{\dot{x}_z(t)}^2$ and $p=2\alpha\lambda\ge 0$. Fix $\tau \in (0,\tau_2) \cap \big(0,T^\lambda(z)\big)$ and $t \in \big(\tau,T^\lambda(z)\big)$. The definition of the restarting time implies that
\[tg'(t) +pg(t) \geq 0.\]
Integrating over $[\tau,t]$, we obtain
$$
t^{p}g(t) \geq \tau^{p}g(\tau). 
$$
Since $t<T^\lambda(z)$, the definition of the restarting time, together with \eqref{E:restarting_times_2}, yield
\begin{equation} \label{E:upper_bound_1}
    \dfrac{d}{dt}\phi(x_z(t)) \leq - \parenc{\dfrac{\alpha}{t}(1-\lambda) + \beta\mu}g(t)\le - \tau^{p}g(\tau) \parenc{\dfrac{\alpha(1-\lambda)}{t^{p+1}} + \dfrac{\beta\mu}{t^p}}.
\end{equation}
Integrating \eqref{E:upper_bound_1} over $[\tau,T^\lambda(z)]$, we deduce that
$$\phi\paren{x_z(T^\lambda(z))} - \phi(x_z(\tau)) \leq -\tau^pg(\tau)\parenc{\int_\tau^{T^\lambda(z)}\dfrac{\alpha(1-\lambda)}{t^{p+1}}dt + \int_\tau^{T^\lambda(z)}\dfrac{\beta\mu}{t^p}dt}.$$
As $t\mapsto\phi\big(x_z(t)\big)$ is nonincreasing on $[0,T^\lambda(z)]$, the P\L\ Inequality \eqref{eq:strong1} gives
\[2\mu (\phi(x_z(\tau)) - \phi\paren{x_z(T^\lambda(z))}\le 2\mu\big(\phi(z) - \phi^*)\le \norm{\nabla\phi(z)}^2.\]
Combining the last two inequalities, and using \eqref{E:reverse}, it follows that
\begin{equation}\label{eq:ineq_upperbound}
\int_\tau^{T^\lambda(z)}\dfrac{\alpha(1-\lambda)}{t^{p+1}}dt + \int_\tau^{T^\lambda(z)}\dfrac{\beta\mu}{t^p}dt\le \frac{\|\nabla \phi(z)\|^2}{2\mu\tau^pg(\tau)}\le \frac{(\alpha+1)^2H(\tau)^2}{2\mu\tau^{p+2}\big(2H(\tau)-1\big)^2}=:M(\tau).
\end{equation}
The quantity $M(\tau)$ is finite for every $\tau\in(0,\tau_2)$, and does not depend on $z$. If $p=0$ ($\lambda=0$), then
$$\alpha\ln\left(\frac{T^\lambda(z)}{\tau}\right)+\beta\mu\big(T^\lambda(z)-\tau\big) \le M(\tau),$$
Thus
$$T^\lambda(z)\le \min\left\{\tau \exp\left(\frac{M(\tau)}{\alpha}\right),\tau+\frac{M(\tau)}{\beta\mu}\right\}.$$
If $p\in(0,1)$, then
$$\frac{\alpha(1-\lambda)}{p}\left[\frac{1}{\tau^p}-\frac{1}{T^\lambda(z)^p}\right]+\frac{\beta\mu}{1-p}\big[T^\lambda(z)^{1-p}-\tau^{1-p}\big]\le M(\tau),$$
so
$$T^\lambda(z)\le \left[\frac{(1-p)M(\tau)}{\beta\mu}\right]^{\frac{1}{1-p}}.$$
If $p=1$ ($\lambda=\frac{1}{2\alpha}\le 1$), then
$$\left(\alpha-\frac{1}{2}\right)\left[\frac{1}{\tau}-\frac{1}{T^\lambda(z)}\right]+\beta\mu\ln\left(\frac{T^\lambda(z)}{\tau}\right) \le M(\tau),$$
whence
$$T^\lambda(z)\le \tau \exp\left(\frac{M(\tau)}{\beta\mu}\right).$$
Since the expressions on the right-hand side do not depend on $z$, the result follows.
\end{proof}

\begin{remark}
If $p=2\alpha\lambda>1$, the procedure employed in the proof above does not yield an upper bound for $T(z)$. Whether or not it is possible to obtain such a bound by a different method is an open question.
\end{remark}

\section{Uniform function value decrease and conclusion} \label{sec:main_theo}

In order to show that the $\lambda$-extended speed restart is proper, it remains to establish property iii) in the definition given in \ref{SS:schemes}. To this end, set
\[\Psi(t) = \parenc{2 - \dfrac{1}{H(t)}}^2.\]
We have the following:

\begin{proposition}\label{p:values_decrease} 
Let $x_z$ be the lazy solution of \eqref{eq:din_avd} starting from $z \notin \argmin(\phi)$. For each $\tau \in (0,\tau_2)\cap(0,T^\lambda(z)]$ and each $t \in \parenc{\tau,T^\lambda(z)}$, we have
\[\phi(x_z(t)) - \phi^* \leq  \left(1-\left[\frac{\alpha(1-\lambda)}{2}+\frac{\beta\mu\tau}{3}\right]\dfrac{2\mu\tau^2\Psi(\tau)}{(\alpha+1)^2}\right)(\phi(z) - \phi^*).\]
\end{proposition}

\begin{proof} Let $\tau \in (0,\tau_2) \cap (0,T^\lambda(z))$, and take $s \in (0,\tau)$. Using \eqref{E:restarting_times_2} and then \eqref{E:reverse}, we deduce that
\begin{eqnarray*}
    \dfrac{d}{ds}\phi(x_z(s)) 
& \le &-\left[\dfrac{\alpha(1-\lambda)}{s} + \beta\mu\right]\norm{\dot{x}_z(s)}^2 \\
& \le & -\left[\alpha(1-\lambda)s + \beta\mu s^2\right]
\Psi(s)\dfrac{\norm{\nabla\phi(z)}^2}{(\alpha+1)^2}.
\end{eqnarray*}
Since $H$ is nonincreasing and larger than 1/2, $\Psi$ is nonincreasing. Integration on $(0,\tau)$ yields 
\begin{eqnarray*}
    \phi(x_z(\tau)) - \phi(z) 
    & \le & -\dfrac{\norm{\nabla\phi(z)}^2}{(\alpha+1)^2}\int_0^\tau\left[\alpha(1-\lambda)s + \beta\mu s^2\right]\Psi(s)\,ds \\
    & \le & -\Psi(\tau)\dfrac{\norm{\nabla\phi(z)}^2}{(\alpha+1)^2}\int_0^\tau\left[\alpha(1-\lambda)s + \beta\mu s^2\right]\,ds \\
    & = & -\Psi(\tau)\dfrac{\norm{\nabla\phi(z)}^2}{(\alpha+1)^2}\left[\frac{\alpha(1-\lambda)\tau^2}{2}+\frac{\beta\mu\tau^3}{3}\right].
\end{eqnarray*}
Using the P\L\ Inequality \eqref{eq:strong1}, and the fact that $\phi(x(t)) \leq \phi(x(\tau))$, we deduce that
\begin{eqnarray*}
    \phi(x_z(t)) - \phi^* 
    & \le & \Big(\phi\big(x_z(\tau)\big) - \phi(z)\Big)  +\big(\phi(z)-\phi^*\big) \\
    & \le & \big(\phi(z)-\phi^*\big) -\Psi(\tau)\dfrac{\norm{\nabla\phi(z)}^2}{(\alpha+1)^2}\left[\frac{\alpha(1-\lambda)\tau^2}{2}+\frac{\beta\mu\tau^3}{3}\right]\\
    & \le & \big(\phi(z)-\phi^*\big) -\dfrac{2\mu\tau^2\Psi(\tau)\big(\phi(z)-\phi^*\big)}{(\alpha+1)^2}\left[\frac{\alpha(1-\lambda)}{2}+\frac{\beta\mu\tau}{3}\right]\\
    & \le &  \left(1-\left[\frac{\alpha(1-\lambda)}{2}+\frac{\beta\mu\tau}{3}\right]\dfrac{2\mu\tau^2\Psi(\tau)}{(\alpha+1)^2}\right)\big(\phi(z)-\phi^*\big),
\end{eqnarray*}
as stated. 
\end{proof}

The results developed above allow us to establish the following:

\begin{proposition} \label{P:conclusion}
Let $\alpha,\beta>0$. For every $\lambda\in[0,\frac{1}{2\alpha}]$, the $\lambda$-extended restarting scheme is proper.    
\end{proposition}

\begin{proof}
    It suffices to set $\tau_*=\bm T_*$ (Corollary \ref{C:lower_bound}), $\tau^*=\bm T^*$ (Proposition \ref{p:upperbound_existence}), and $$Q=1-\left[\frac{\alpha(1-\lambda)}{2}+\frac{\beta\mu\bm T_*}{3}\right]\dfrac{2\mu(\bm T_*)^2\Psi(\bm T_*)}{(\alpha+1)^2}$$ 
    (Proposition \ref{p:values_decrease}).
\end{proof}
\section{Numerical Illustrations}\label{sec:numerical}
In this section, we provide numerical evidence of the improvement achieved by the extended speed restart. Inspired by our results in continuous time, we propose a new restarting scheme for accelerated gradient methods and carry out illustrative tests on its performance.

\subsection{Continuous case}
Let us consider a quadratic function 
 $\phi:\R^n \to \R$ of the form 
 \begin{equation}\label{eq:func_cont}
     \phi(x) = \frac{1}{2}x^T\mathcal{A}x + bx,
 \end{equation}
where $\mathcal{A}$ is a real positive-definite matrix of size $n\times n$, and $b \in \R^n$. By using its gradient and Hessian in \eqref{eq:din_avd}, we obtain a $n$-dimensional system of differential equations given by
\[\ddot{x}(t) + \dfrac{\alpha}{t}\dot{x}(t) + \beta \mathcal{A}\dot{x}(t)  + \mathcal{A}x(t) + b =0.\]
We set $n=40$, we generate $\mathcal{A}$ and $b$ randomly (taking into consideration that $\mathcal{A}$ must be positive), and we solve \eqref{eq:din_avd} using the\texttt{
Python} tool \texttt{solve\_ivp} from the package \texttt{scipy}. We use $\alpha=3$, $\beta=1$, and we solve on initial time $t_0=1$ with initial condition $x_0 \in \R^n$ defined as $1$ in all of its entries and zero initial velocity. Figure \ref{fig:continuous_lambdas} presents the function values along the trajectories minus the minimum value of the function, which can be easily computed, and the continuous restart routine for different choices of $\lambda$. The case $\lambda=0$, depicted as a red line, is the speed restart studied in \cite{maulen2023speed}. The Figure shows a faster linear convergence for larger values of $\lambda$. Table \ref{tab:coefs_cont} shows the coefficients of the approximation $\phi\big(x(t)\big) - \phi^*\sim Ae^{-Bt}$ for the restarted trajectories. 

\begin{figure}[h]
    \centering    \includegraphics[width=0.75\linewidth]{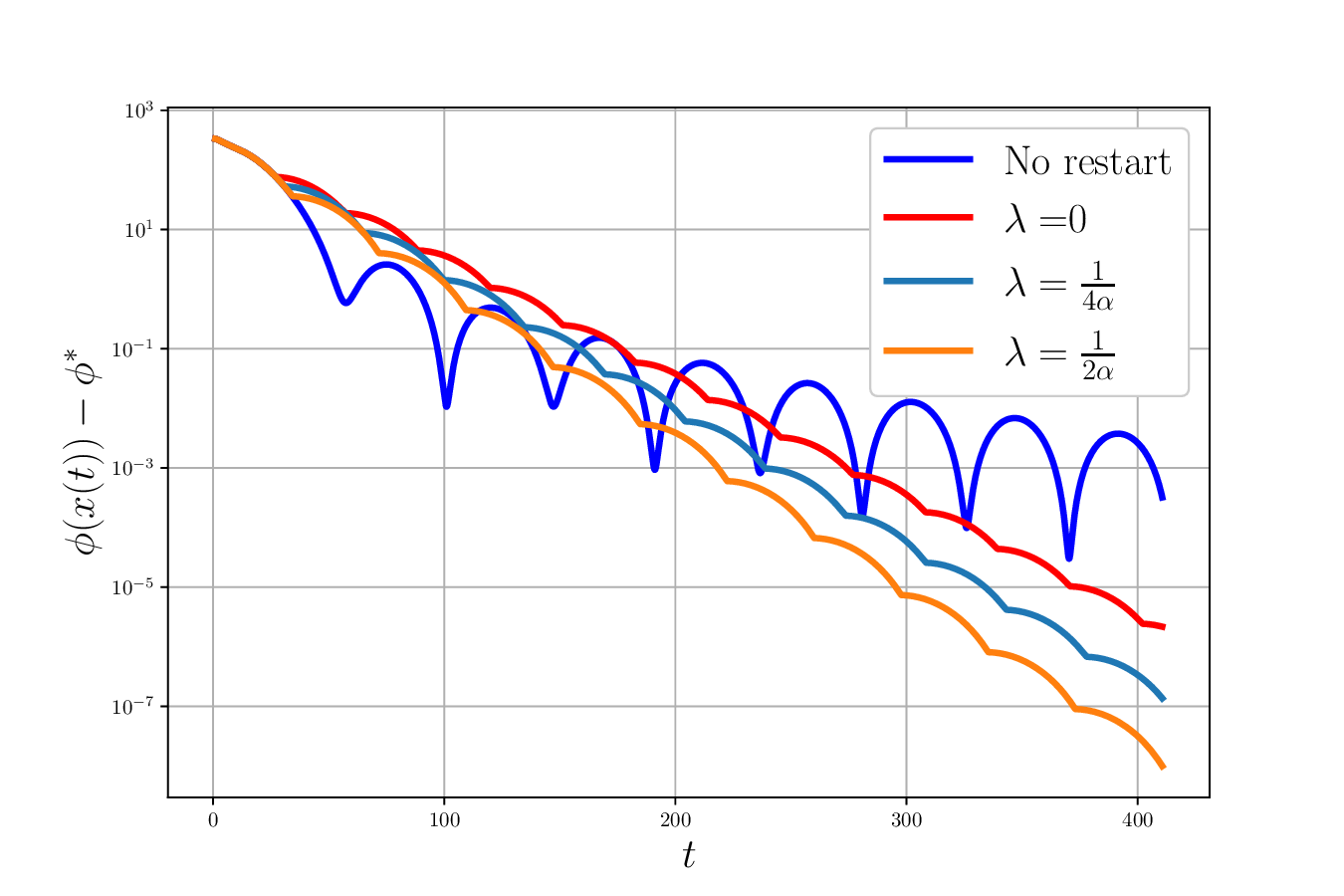}
    \caption{Values along the trajectory, for the solutions of \eqref{eq:din_avd} with $\phi$ a quadratic function as in \eqref{eq:func_cont}, for different choices of $\lambda$.}
    \label{fig:continuous_lambdas}
\end{figure}

\begin{table}[ht]
    \centering
    \begin{tabular}{cccc}
    \toprule
    $\lambda$ & 0& $1/4\alpha$&$1/2\alpha$ \\
    \midrule
       $A$ & 3.3093e+02  & 3.6454e+02 &  3.2467e+02 \\
        $B$ & 4.5618e-02 & 5.2021e-02 & 5.6794e-02 \\
    \bottomrule
    \end{tabular}
    \caption{Coefficients of the approximation $\phi\big(x(t)\big) - \phi^*\sim Ae^{-Bt}$, for the function values along the solutions of \eqref{eq:din_avd} with $\phi$ a quadratic function as in \eqref{eq:func_cont}. }
    \label{tab:coefs_cont}
\end{table}

\subsection{A proposed heuristic for a discrete criterion}\label{sec:discrete}

Finite time discretizations of equation \eqref{eq:din_avd} leads to the class of first-order algorithms including a gradient correction term (see \cite{attouch2020first}), which are designed to build minimizing sequences for the function. In this section, we propose the following criterion obtained by a discretization of the restarting criterion in \eqref{E:ESRestart}:
\begin{equation}
    \norm{\Delta_{k+1}}^2 \leq \paren{1 - \dfrac{2\alpha\lambda}{k}}\norm{\Delta_k}^2,
\end{equation}
where $\Delta_k = x_k - x_{k-1}$ and $\lambda \in [0,\frac{1}{2\alpha}]$. For a stepsize $h>0$, we consider a discretization of \eqref{eq:din_avd} leading to 
\begin{equation}\label{eq:algorithm}
    \left\lbrace\begin{array}{rl}
  y_k   & = x_k + \alpha_k(x_k-x_{k-1}) - \beta h(\nabla \phi(x_k) - \nabla \phi(x_{k-1})) \\
   x_{k+1}  & = y_k - h^2 \nabla \phi(y_k),
\end{array}\right.
\end{equation}
with $\alpha_k =\frac{k}{k+\alpha} = 1 - \frac{\alpha}{k+\alpha}$, which is an usual choice for Nesterov's acceleration (see, for instance, \cite{wang2025acceleratedgradientmethodsinertial}). This method can be interpreted as Nesterov's accelerated method  including a \textit{gradient correction} term when $\beta>0$. This scheme is closely related to the family of algorithms studied in \cite{attouch2020first}, where the parameters are required to satisfy $\alpha \geq 3$, $h^2 L \leq 1$, and $\beta \in (0,2h)$. In particular, these conditions enforce $\beta$ to remain small in practice, so that the resulting sequences exhibit a comparable behavior to that of Nesterov’s accelerated gradient method (NAG). The latter converges at a linear \cite{li2024linear,wang2025fast}---but not accelerated---rate. As we shall see below, the restarting schemes considerably improve these algorithms' performance.

Methods including restarting routines usually include the use of a $k_{\min}$ parameter which prevents the algorithm on performing the first restart too soon (see \cite{SBC2016,maulen2023speed}). This is not necessary with our proposed method. 

\begin{algorithm}[h]
\SetAlgoLined
Choose $x_0=x_{-1} \in \R^n$, $N \in \N$, $\alpha\geq 1$, $\beta>0$, $h>0$ and $\lambda \in \paren{0,\frac{1}{2\alpha}}$. Set $j=1$.\\
 \For{$k=1 \ldots N$}{
  \text{Compute} $y_k= x_k + \left(1-\frac{\alpha}{j+\alpha}\right)(x_k-x_{k-1}) - \beta h(\nabla \phi(x_k) - \nabla \phi(x_{k-1}))$, \\
  \text{and then} $x_{k+1} = y_k - h^2 \nabla \phi(y_k)$. \\
  \uIf{$\norm{x_{k+1}-x_k}^2 < \paren{1 - \dfrac{2\alpha\lambda}{j}}\norm{x_k - x_{k-1}}^2$}{$j=1$;}
  \uElse{$j=j+1$.}
 }
\Return $x_N$.
 \caption{Inertial Gradient Algorithm with Hessian Damping (IGAHD) - Extended Speed Restart Version}
  \label{algorithm:restart}
\end{algorithm}


\subsubsection{Illustrative examples}
The performance of Algorithm \ref{algorithm:restart} is tested for different choices of $\phi$, and the linear convergence becomes apparent by plotting $\phi(x_k)-\phi^*$ in logarithmic scale. In each case, the blue line corresponds to the non restarted algorithm, and the red line corresponds to the case $\lambda=0$, for comparison with \cite{maulen2023speed}. The linear convergence exhibited by the speed restart scheme is improved by considering the extended criteria with $\lambda > 0$. Then, we perform for each case an approximation of the function values decay as $\phi\big(x_k\big) - \phi^*\sim Ae^{-Bk}$ via linear regression. 
\begin{itemize}
    \item \textbf{An ill-posed problem:} We consider $\phi$ as in Example \ref{eq:fun_example} with $\rho =10$. We set $\alpha=3$, $\beta=h=1/\sqrt{L}$.  We test Algorithm \ref{algorithm:restart} with $\lambda \in \set{0,\frac{1}{4\alpha},\frac{1}{2\alpha}}$. The plot of $\phi(x_k) - \phi^*$ is presented in Figure \ref{fig:lam_diff_qua} and the coefficients of the approximation in Table \ref{tab:coefs_discrete}. 
    \item \textbf{Quadratic function:} Consider the function 
\[\phi(x) = \frac{1}{2}x^T\mathcal{A}x + bx,\]
where $\mathcal{A}$ is a positive definite symmetric $500\times500$ matrix, generated with eigenvalues in (0,1) and $b \in \R^{500}$ as a randomly generated vector with i.i.d. Gaussian entries with mean 0 and variance 1. 
We set $\alpha=3$ and $\beta=h=1/\sqrt{L}$, and the initial point $x_{0}=x_{1}$ is generated randomly. Figure \ref{fig:lam_diff_matr} presents the evolution of $\phi(x_{k})-\phi^{*}$ along the iterations and the coefficients of the approximation in Table \ref{tab:coefs_discrete}. 
\item \textbf{Log-sum-exp:} We consider the following function which is convex but \textit{not} strongly convex
\[\phi(x)=\rho \operatorname{log}\left[\sum_{i=1}^{m}\operatorname{exp}((a_{i}^{T}x-b_{i})/\rho)\right].\]
Here, we set the smoothness parameter $\rho=10$, $n=50$, $m=20$ and  $a_i \in \R^n$ are randomly generated vectors with i.i.d. standard Gaussian entries, for $i=1,\ldots,m$, and $b=(b_{i}) \in \R^n$ has i.i.d. Gaussian entries with mean 0 and variance 1.  We set $\alpha=3$, $\beta=h=1/\sqrt{L}$, and the initial point $x_{0}=x_{1}$ is generated randomly as well. Figure \ref{fig:lam_diff_log} shows the evolution of $\phi(x_{k})-\phi^{*}$ along the iterations and the coefficients of the approximation are presented in Table \ref{tab:coefs_discrete}.
\end{itemize}

\begin{figure}[htbp]
\centering
\subfigure[ $\phi(x) = \frac{1}{2}(x_1^2+\rho x_2^2+\rho^2 x_3^2)$.]{\includegraphics[width=0.32\textwidth]{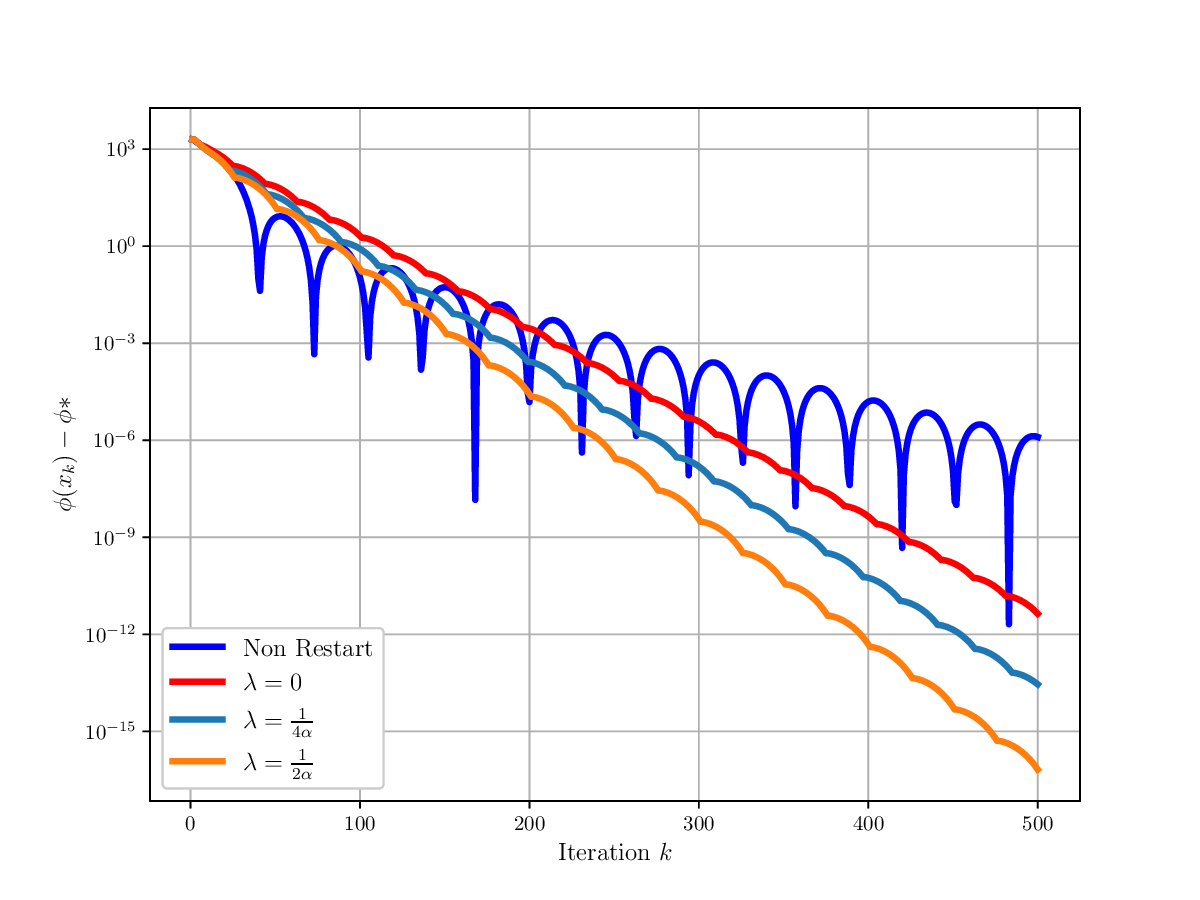}\label{fig:lam_diff_qua}}
	\subfigure[$\phi(x) = \frac{1}{2}x^T\mathcal{A}x + bx$.]{\includegraphics[width=0.32\textwidth]{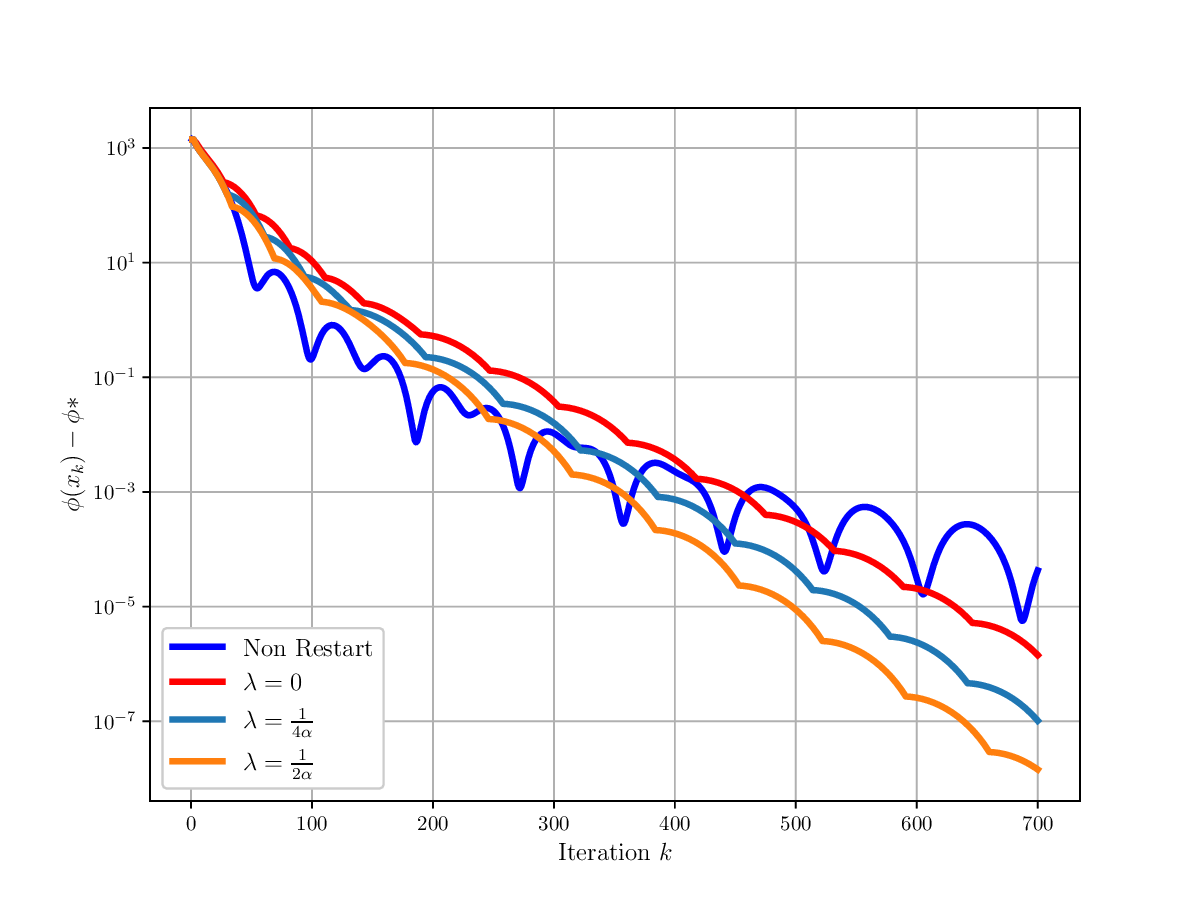}\label{fig:lam_diff_matr}}
	\subfigure[$\phi(x)=\rho \ln(\sum\exp(\frac{a_{i}^Tx-b_{i}}{\rho}))$.]{\includegraphics[width=0.32\textwidth]{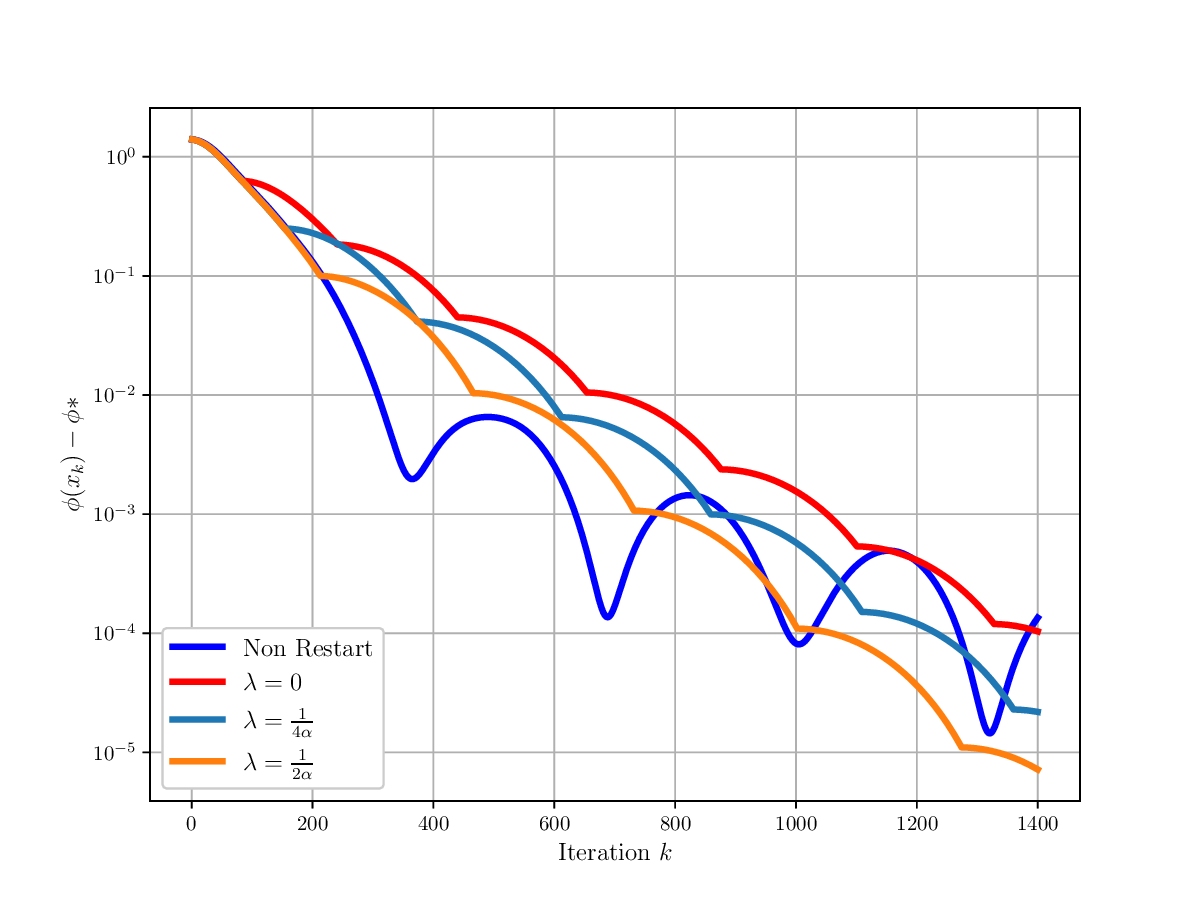}\label{fig:lam_diff_log}}
	\caption{Values of $\phi(x_{k})-\phi^{*}$ along the iterations for non restart, speed restart scheme ($\lambda=0$) and extended restart scheme with different choices of $\lambda$.}
\label{fig:discrete time}
\end{figure}

\begin{table}[htbp]
    \centering
    \begin{tabular}{ccccc}
    \toprule
    \multirow{2}{*}{$\phi(x)$}& & \multicolumn{3}{c}{$\lambda$} \\
    \cmidrule(lr){3-5}
    & & 0& $1/4\alpha$&$1/2\alpha$ \\
    \midrule
     \multirow{2}{*}{$\frac{1}{2}\paren{x_1^2+10x_2^2+100x_3^2}$}  &$A$ & 1.949e+02  & 1.721e+02 &  1.921e+02 \\
      &$B$  & 6.711e-02 & 7.746e-02 & 8.911e-02 \\
      \midrule
      \multirow{2}{*}{$\frac{1}{2}x^T\mathcal{A}x + bx$}&$A$ & 1.873e+02  & 1.708e+02 &  1.582e+02 \\
       & $B$ & 2.696e-02 & 3.056e-02 & 3.374e-02 \\
       \midrule
      \multirow[c]{2}{*}{$\rho \log\left(\ds\sum_{i=1}^{m}\exp((a_{i}^Tx-b_{i})/\rho)\right)$} &$A$ & 1.186 & 1.080 &  9.831e-01 \\ [2mm]
       &$B$ & 6.771e-03 & 7.714e-03 & 8.660e-03 \\[2mm]
    \bottomrule
    \end{tabular}
    \caption{Coefficients of the approximation $\phi\big(x_k\big) - \phi^*\sim Ae^{-Bk}$, for the different choices of $\phi$ presented in Section \ref{sec:discrete} over the realizations of Algorithm \ref{algorithm:restart}.}
    \label{tab:coefs_discrete}
\end{table}

Despite the lack of a theoretical convergence guarantee, restarting schemes based on function values show remarkable numerical performance \cite{o2015adaptive}. We propose to use the first function value restart point as a warm start. The evolution of the algorithm with and without the warm start is displayed in Figure \ref{fig:discrete_warm} for $\lambda=\frac{1}{2\alpha}$. The coefficients of the approximation $\phi\big(x_k\big) - \phi^*\sim Ae^{-Bk}$, for different values of $\lambda$, are presented in Table \ref{tab:coefs_discrete_warm}.

\begin{figure}[htbp]
\centering
\subfigure[$\phi(x) = \frac{1}{2}(x_1^2+\rho x_2^2+\rho^2 x_3^2)$.]{\includegraphics[width=0.32\textwidth]{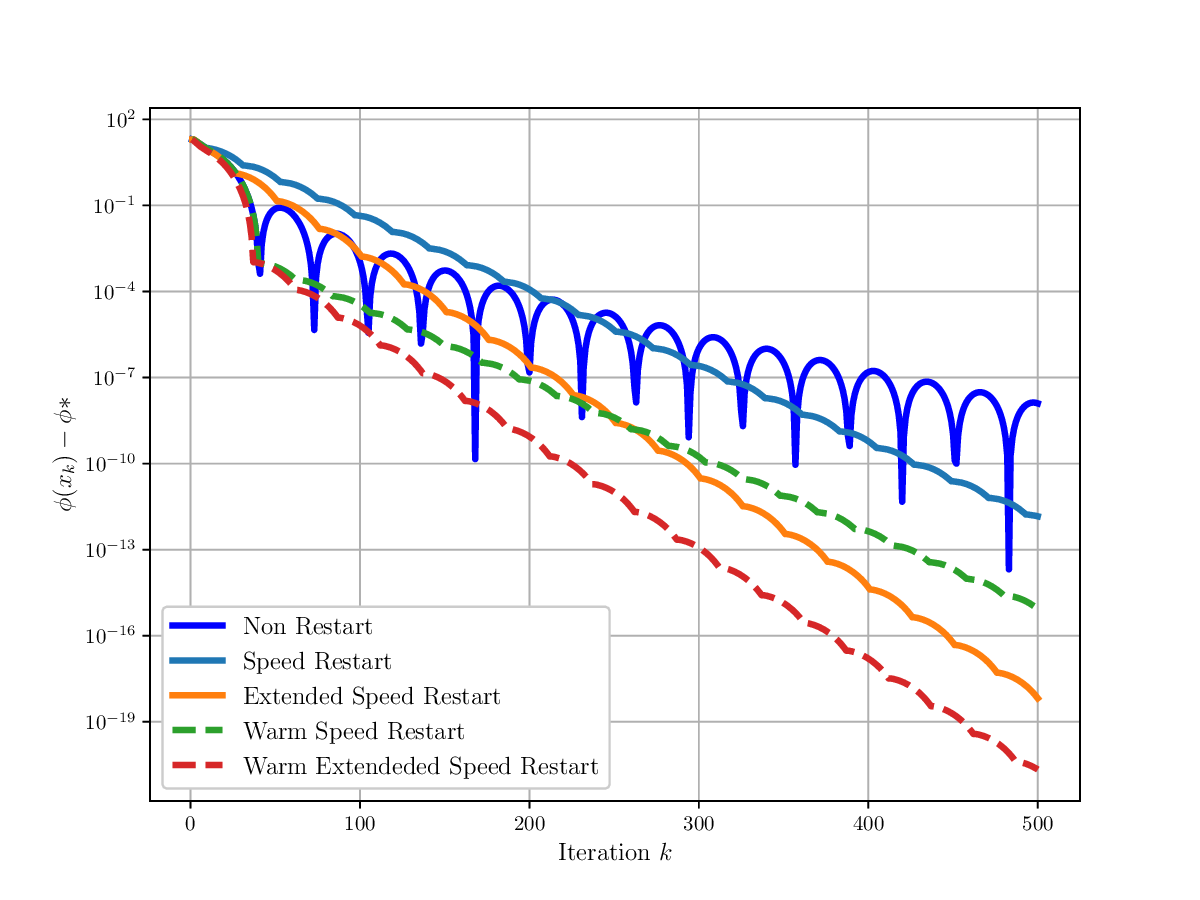}\label{fig:warm_quad}}
	\subfigure[$\phi(x) = \frac{1}{2}x^T\mathcal{A}x + bx$.]{\includegraphics[width=0.32\textwidth]{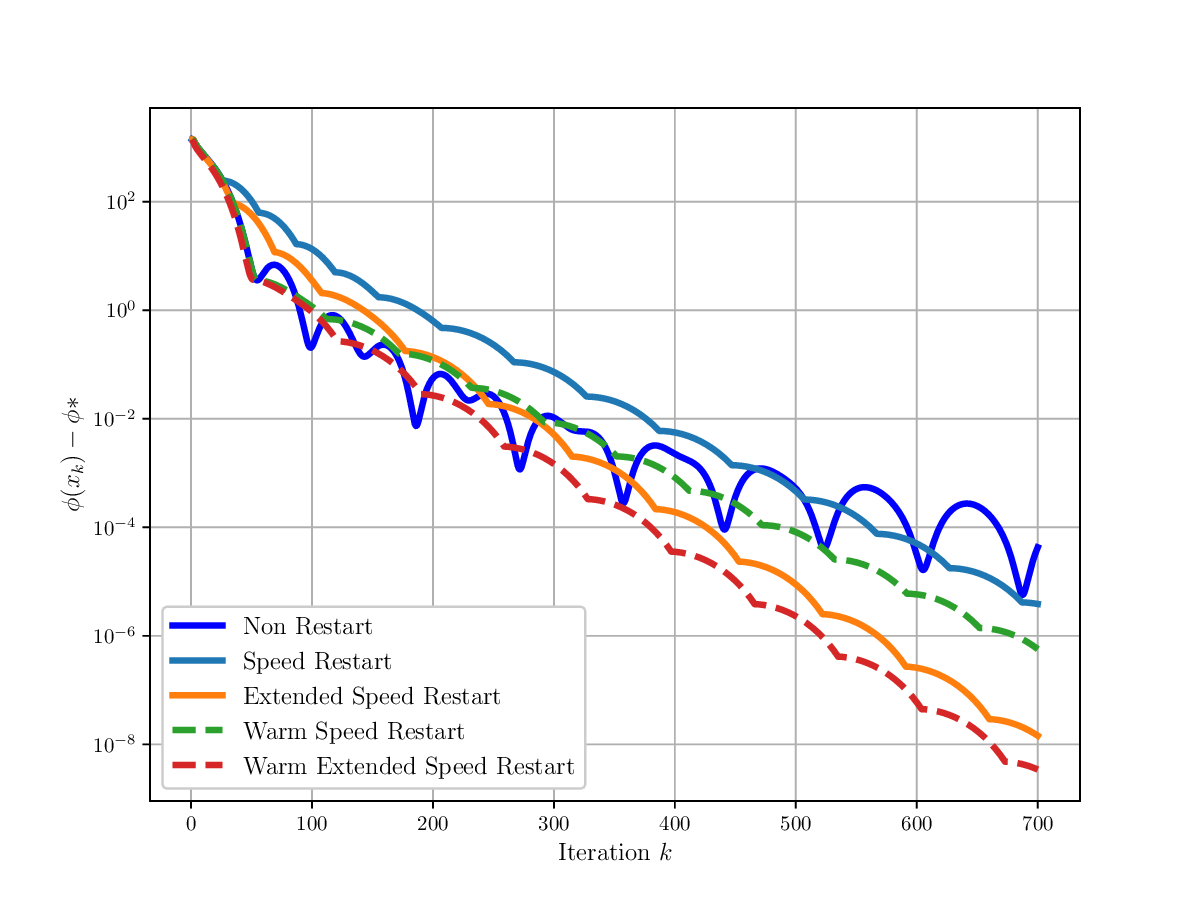}\label{fig:warm_matr}}
	\subfigure[$\phi(x)=\rho \ln(\sum\exp(\frac{a_{i}^Tx-b_{i}}{\rho}))$.]{\includegraphics[width=0.32\textwidth]{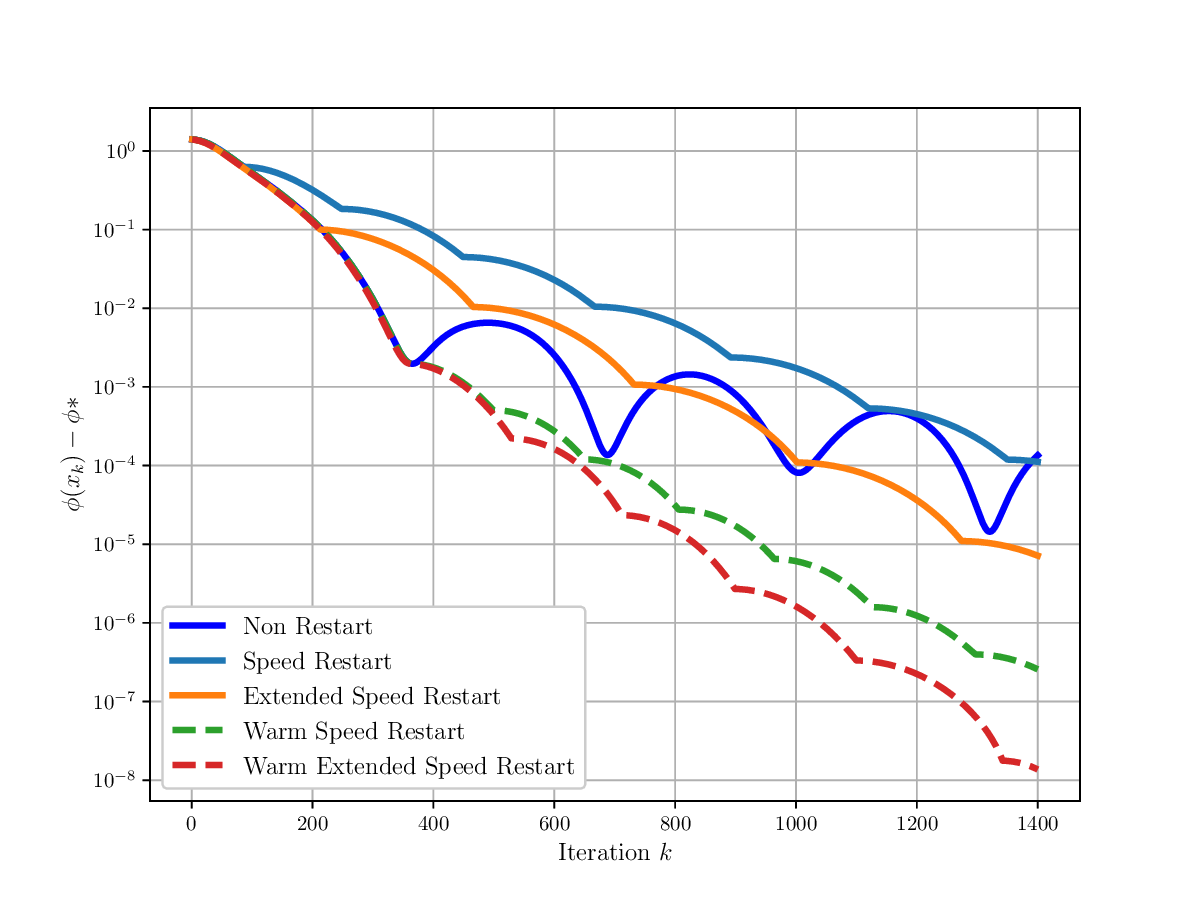}\label{fig:warm_log}}
	\caption{Values of $\phi(x_{k})-\phi^{*}$ along the iterations for non restart, speed restart scheme ($\lambda=0$), extended speed restart and the versions with the warm start.}
\label{fig:discrete_warm}
\end{figure}
\begin{table}[htbp]
    \centering
    \begin{tabular}{ccccc}
    \toprule
    \multirow{2}{*}{$\phi(x)$}& & Warm & Extended  & Warm Extended \\
    & & Speed Restart & Speed Restart & Speed Restart \\
     \midrule
     \multirow{2}{*}{$\frac{1}{2}\paren{x_1^2+10x_2^2+100x_3^2}$}  & $A$ & 1.391e-02   & 1.921e+01 & 4.054e-02\\
    &$B$ & 6.064e-02 & 8.911e-02   & 8.907e-02 \\
      \midrule
      \multirow{2}{*}{$\frac{1}{2}x^T\mathcal{A}x + bx$}& $A$ & 1.374e+01   & 1.582e+02  & 2.100e+01\\
    &$B$ & 2.430e-02 & 3.374e-02   & 3.243e-02\\
       \midrule
      \multirow[c]{2}{*}{$\rho \log\left(\ds\sum_{i=1}^{m}\exp((a_{i}^Tx-b_{i})/\rho)\right)$} & $A$ & 5.476e-02   & 9.825e-01 & 1.321e-01\\ [2mm]
    &$B$ & 9.002e-03 &8.659e-03   & 1.139e-02 \\ [2mm]
    \bottomrule
    \end{tabular}
    \caption{Coefficients of the approximation $\phi\big(x_k\big) - \phi^*\sim Ae^{-Bk}$, for the different choices of $\phi$ presented in Section \ref{sec:discrete} over the realizations of Algorithm \ref{algorithm:restart} with a warm start.}
    \label{tab:coefs_discrete_warm}
\end{table}

\subsubsection{Kernel Regularized Learning}
In this section, we present our results in the context of Kernel Regularized Learning \cite{hofmann2008kernel,mendelson2010regularization,scholkopf2002learning,yan2023confidence}. In this framework, we aim to estimate a function $\phi$ by solving the optimization problem
\[\min_{\phi \in\mathcal{H}} \sum_{i=1}^{n}\ell(b_i,\phi(a_i)) + \frac{\theta}{2}\left\Vert \phi \right\Vert^2_{\mathcal{H}},\]
where $(a_i,b_i),$ $i=1,\ldots,n$ are i.i.d. observations drawn from a certain distribution, $\ell$ is a loss function, $\mathcal{H}$ is a reproducing Kernel Hilbert space associated with a Kernel function $\mathcal{K}$ and $\theta>0$ is a regularization parameter.
As shown in \cite{sain1996nature,wahba1990spline}, the solution of this problem lies in the span of the kernel functions evaluated at the training points. Consequently, the problem can be rewritten as the finite-dimensional optimization problem
\begin{equation}\label{eq:prob_kernel}
    \min_{x \in \R^n} \sum_{i=1}^{n}\ell(b_i,\left( \mathcal{K}x\right)_i) + \frac{\theta}{2} x^\top \mathcal{K}x.
\end{equation}
The main challenges in solving problem \eqref{eq:prob_kernel} arise from the structure of the loss function $\ell$ and, more importantly, from the potentially large sample size $n$. In this section, we intend to replicate the experiments in \cite{heng2025inertial}, where the authors address this difficulty by using random projection techniques, to restrict the solution to an $m$-dimensional space, with $m\ll n$. Let us consider $G \in \R^{m\times n}$ as a random projection matrix, and the problem to solve becomes
\begin{equation}\label{eq:prob_kernel_m}
    \min_{x \in \R^m} \sum_{i=1}^{n}\ell\left(b_i,\left( \mathcal{K}G^\top x\right)_i\right) + \frac{\theta}{2} x^\top G\mathcal{K}G^\top x.
\end{equation}
In order to replicate the experiments in \cite{heng2025inertial}, we design the \textit{sketching matrix} $G$ using the Nystr{\"o}m method \cite{rudi2015less,williams2000using}: as uniformly sampled rows from the $n \times n$ identity matrix. We focus on two problems defined by the loss function $\ell$: Kernel logistic regression and Kernel multinomial regression.
\begin{itemize}
    \item \textbf{Kernel logistic regression: } in this context, we aim to minimize the objective
\[\phi(x) = -\sum_{i=1}^{n}\left[b_i \log(p_i) + (1-b_i)\log(1-p_i)\right] + \frac{\theta}{2} x^\top G\mathcal{K}G^\top x,\]
where 
\[p_i = \dfrac{1}{1+\exp\left( -(\mathcal{K}G^\top x )_i\right)}, \quad b_i \in \left\lbrace 0,1\right\rbrace.\]
The gradient can be explicitly computed as
\[\nabla \phi(x) = -G\mathcal{K}(b-p) + \theta G \mathcal{K} G^\top x.\]
\item \textbf{Kernel multinomial regression: } in this context, $b_i \in \left\lbrace 1,\ldots,q \right\rbrace$ indicates the category of sample $i$, and using the $q$-th category as the reference category, the probability $p_{ij}$ that $b_i=j$ is given by
    \[ p_{ij}(X) = \left\lbrace\begin{array}{rl}
       \dfrac{\exp(\eta_{ij})}{1+\sum_{k=1}^{q-1}\exp(\eta_{ik})}  & i\leq j < q  \\ [5mm]
       \dfrac{1}{1+\sum_{k=1}^{q-1}\exp(\eta_{ik})}  & j=q,
    \end{array}\right.\]
    where $\eta_{ij} = \left( \mathcal{K}G X_j\right)_i$, and $X_j$ is the $j-$th column of the $m\times(q-1)$ coefficient matrix $X$. Then, the objective function is 
    \[\phi(X) = - \sum_{i=1}^{n}\sum_{j=1}^{q} \iota_{\lbrace i=j \rbrace}\log(p_{ij}) + \frac{\theta}{2} \text{tr}\left( X^\top G \mathcal{K} G^\top X\right)\!,\]
    where $\iota$ is the indicator function and $\text{tr}(\cdot)$ denotes the trace of a matrix. Let $B \in \R^{n\times(q-1)}$ and $P \in \R^{n\times(q-1)}$, where each row of $B$ is a binary vector such that $b_{ij}=1$ if $b_i=j$, and each row of $P$ contains the corresponding probabilities $p_{ij}$ for $j=1,\ldots,q-1$. By vectorizing $X$ as $\text{vec}(X)$, the gradient is expressed as 
    \[\nabla \phi(\text{vec}(X)) = - \text{vec}(G\mathcal{K}(B-P)) + \theta \,\text{vec}(G \mathcal{K} G^\top X).\]
\end{itemize}
For the numerical tests, we generate the simulated data such as in \cite{heng2025inertial}, with $n=2^{14}$ , $d=50$ and for the multinomial we consider $q=3$. We use the Radial Basis Function (RBF) kernel $\mathcal{K}(a_1,a_2) = \exp\left\lbrace - \frac{1}{2}\left\Vert a_1 - a_2 \right\Vert^2\right\rbrace$, $\theta=10^{-4}$, and set $G$ to have $m=2^{11}$ rows. Figure \ref{fig:kernel} shows the performance of Algorithm \ref{algorithm:restart} in both contexts, using $\alpha=3$, $\beta = h = 1/\sqrt{L}$, the initial points consisting only of zeros, and different values of $\lambda$ for the restart. We also present one realization of the warm sart for the larger value of $\lambda$. For each case, we compute a reference solution $\phi^*$: for both cases, we use one of the restart methods performing more iterations. 

\begin{figure}[h]
    \centering
    \subfigure[Logistic]{\includegraphics[width=0.48\textwidth]{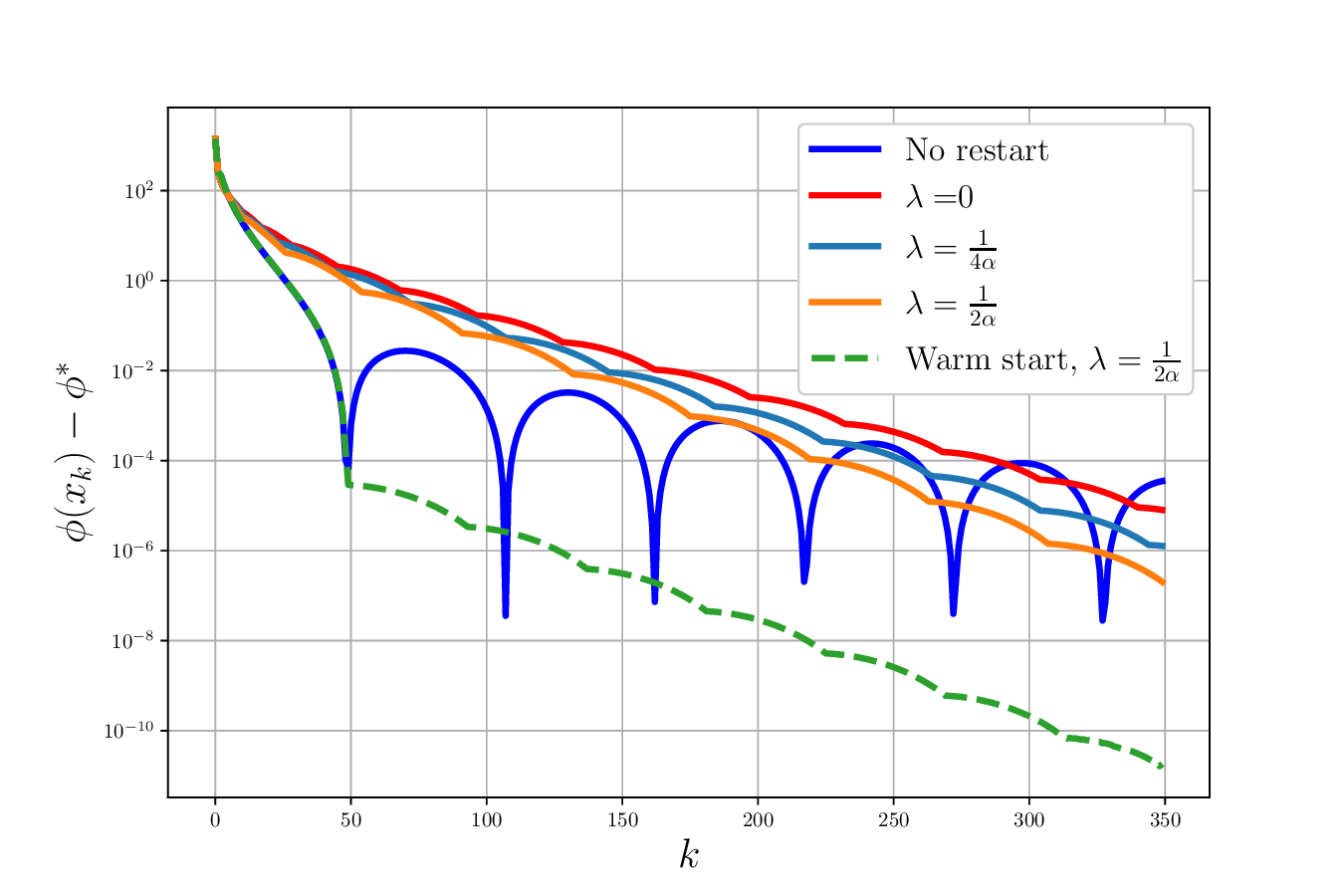}}
    \subfigure[Multinomial]{\includegraphics[width=0.48\textwidth]{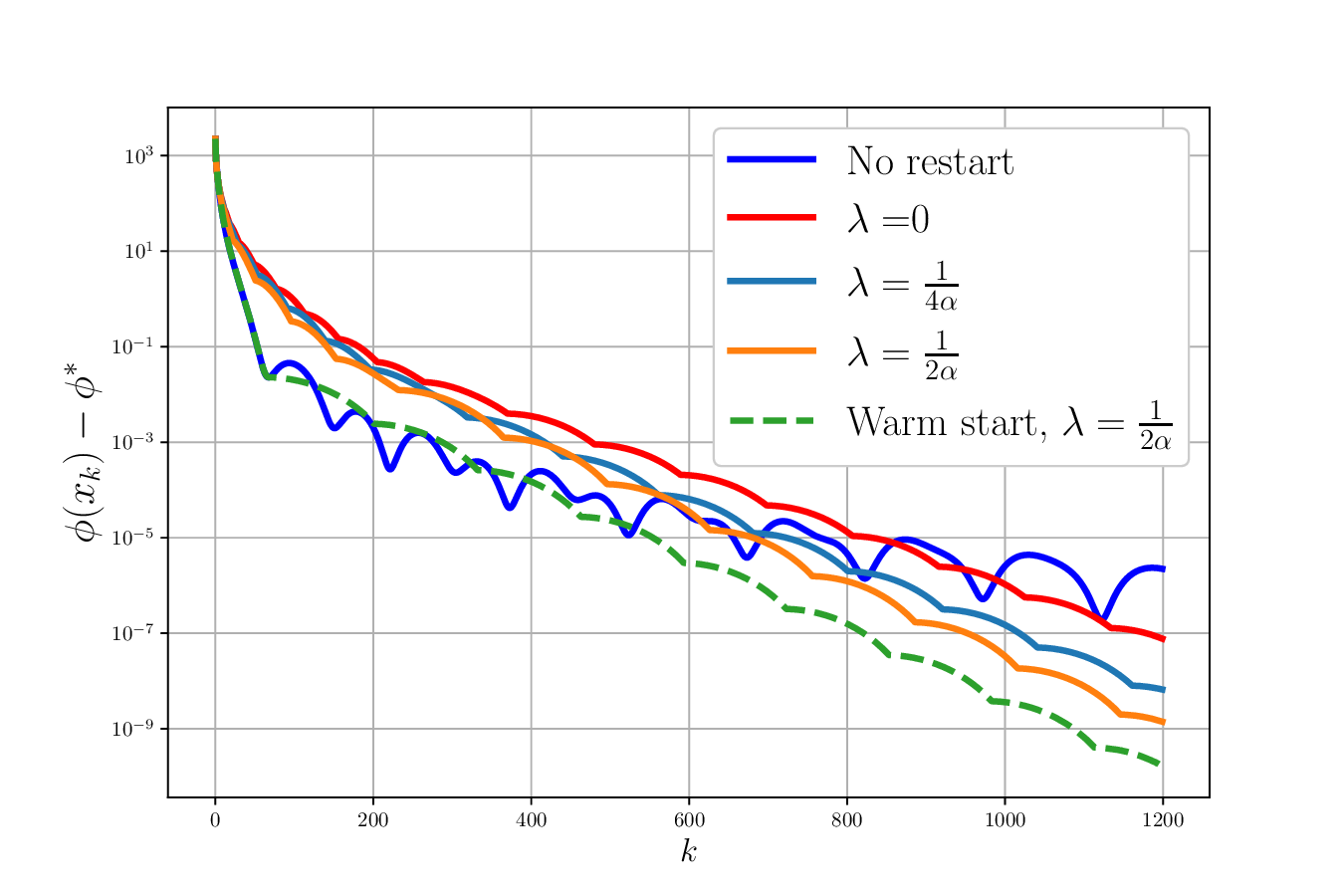}}
    \caption{Kernel regularized learning.}
    \label{fig:kernel}
\end{figure}


\section{Conclusion}\label{sec:conculsion}
We propose a new speed restart scheme for an inertial dynamical system with Hessian-driven damping \eqref{eq:din_avd}. The function values along the trajectories are proven to converge linearly when the objective function is $\mu$-strongly convex. The extended speed restart scheme presented can be interpreted as an interpolation of the speed restart and the function values restart. Numerical simulations indicate a significant improvement over existing restart strategies. \\

Our research opens new questions: First, although the extended speed restart scheme can be implemented for any $\lambda\in[0,1]$, an upper bound for the restarting time has not been established for $\lambda > \frac{1}{2\alpha}$. If $\alpha\ge \frac{1}{2}$, this range includes the important case $\lambda=1$, which matches the function values restart scheme when $\beta=0$. Finding such an upper bound would immediately imply the linear convergence of the function value restart, in view of Lemma \ref{L:abstract_linear}. The convergence rate also depends on the factor $Q$ computed in Proposition \ref{p:values_decrease}, which might be suboptimal. Determining the minimal possible value of $Q$ for which the rate holds across the class of $\mathcal{C}^2$, $\mu$-strongly convex functions with Lipschitz gradients remains an open problem. The nonsmooth setting is also yet to be analyzed. Extending the speed restart scheme to obtain linear convergence for \eqref{eq:hr_dinavd} appears to be significantly more challenging. In contrast with the present setting, the analysis for establishing suitable lower bounds on the restarting times would require a complete different technique. Developing new analytical techniques to handle these bounds and prove convergence under restart schemes remains an open problem. Finally, we do not address the convergence analysis of the heuristic restarting mechanism for algorithms, obtained by discretization of our criterion, which would provide a linear convergence rate guarantee for restarted first-order methods.




\section*{Acknowledgements} {\small The first author (J.J. Maulén) was supported by ANID (Chile) Fondecyt Postdoctoral grant 3250609, and CMM BASAL funds for the Center of Excellence FB210005. This benefited from the support of the FMJH Program Gaspard Monge for optimization and operations research and their interactions with data science. }
\bibliographystyle{abbrv}
\bibliography{referencias}
\end{document}